\documentclass[12pt]{article}
\usepackage{amsmath,epsfig, amsthm,amsfonts,amsbsy,amssymb,latexsym,amsxtra}
\usepackage[usenames]{color}
\usepackage{hyperref}
\usepackage{pst-infixplot}
\usepackage{pst-plot}
\usepackage{pstricks}

\usepackage[sort,comma]{natbib}
\makeatletter \@addtoreset{equation}{section}

\newcommand{\disp}{\displaystyle}

\newcommand{\beq}{\begin{equation}}
\newcommand{\eeq}{\end{equation}}
\newcommand{\bed}{\begin{displaymath}}
\newcommand{\eed}{\end{displaymath}}
\newcommand{\ben}{\begin{eqnarray*}}
\newcommand{\een}{\end{eqnarray*}}

\newcommand{\bedd}{\bed\begin{array}{l}}
\newcommand{\eedd}{\end{array}\eed}
\newcommand{\bea}{\bed\begin{array}{rl}}
\newcommand{\eea}{\end{array}\eed}
\newcommand{\ad}{&\!\!\!\disp}
\newcommand{\aad}{&\disp}
\newcommand{\barray}{\begin{array}{ll}}
\newcommand{\earray}{\end{array}}
\newcommand{\lbar}{\overline}

\def\({\left(}
\def\){\right)}

\def\tr{\hbox{\rm tr}}
\def\one{{\hbox{\rm 1{\kern -0.35em}1}}}

\newcommand{\F}{{\mathfrak F}}
\newcommand{\op}{{\mathcal L}}
\def\cd{(\cdot)}
\newcommand{\M}{{\mathcal M}}
\newcommand{\rr}{{\mathbb R}}
\def\lan{\big\langle}
\def\ran{\big\rangle}
\renewcommand{\SS}{{\mathcal S}}
\newcommand{\la}{\lambda}

\newcommand{\vphi}{\varphi}
\newcommand{\e}{\varepsilon}

\newcommand{\al}{\alpha}
\newcommand{\ga}{\gamma}
\newcommand{\sg}{\sigma}
\newcommand{\tha}{\theta}
\newcommand{\pr}{{\mathbb P}}
\newcommand{\ex}{{\mathbb E}}
\newcommand{\A}{{\mathcal A}}
\newcommand{\Ax}{{\mathcal A}_{x,\alpha}}
\newcommand{\wdt}{\widetilde}

\newcommand{\tu}{{\tilde u}}
\newcommand{\tv}{{\tilde v}}

\newcommand{\V}{\mathbf V}
\newcommand{\xal}{\ensuremath{x,\alpha}}
\newcommand{\xalz}{\ensuremath{x_0,\alpha_0}}
\newcommand{\xz}{\ensuremath{x_0}}
\newcommand{\tvphi}{\tilde \varphi}

\newcommand{\Dj}{D_{x_j}}
\newcommand{\Dij}{D_{x_ix_j}}
\newcommand{\xe}{x_{\varepsilon}}
\newcommand{\ye}{y_{\varepsilon}}
\newcommand{\set}[1]{\left\{#1\right\}}
\newcommand{\abs}[1]{\left\vert #1\right\vert}

\newtheorem{thm}{Theorem}[section]
\newtheorem{prop}[thm]{Proposition}
\newtheorem{lem}[thm]{Lemma}

\theoremstyle{definition}
\newtheorem{rem}[thm]{Remark}
\newtheorem{defn}[thm]{Definition}
\newtheorem{exm}[thm]{Example}

\newcommand{\thmref}[1]{Theorem~{\rm \ref{#1}}}
\newcommand{\lemref}[1]{Lemma~{\rm \ref{#1}}}

\title{On Singular Control Problems with State Constraints and  Regime-Switching: A Viscosity Solution Approach}
\author{Qingshuo Song\thanks{Department of Mathematics, City
    University of Hong Kong, 83 Tat Chee Avenue, Kowloon Tong, Hong
    Kong, {\tt   song.qingshuo@cityu.edu.hk}. 
    } \and Chao Zhu\thanks{Department of Mathematical Sciences, University of Wisconsin-Milwaukee, Milwaukee, WI 53201, USA, {\tt zhu@uwm.edu}. 
    }}
\begin{document}

\maketitle

\begin{abstract}
This paper investigates a singular stochastic control problem for a multi-dimensional regime-switching diffusion process confined in  an unbounded domain. The objective is to maximize the total expected discounted rewards from exerting the singular control. Such a formulation stems from application areas such as optimal harvesting multiple species and optimal dividends payments schemes in random environments. With the aid of weak dynamic programming principle, we characterize the value function to be the unique constrained viscosity solution of a certain system of coupled nonlinear quasi-variational inequalities. Several examples are analyzed in details to demonstrate the main results. 


\end{abstract}

\setlength{\baselineskip}{0.24in}
\section{Introduction}\label{sect-intro}
This paper is concerned with a class of singular stochastic control problems with state constraints. 
The controlled regime-switching diffusion process $X$ and the singular control process $Z$ take values in a convex cone $S\subset \rr^n$.  The control problem has the  state process
\begin{displaymath}
X(t)= x + \int_0^t b(X(s),\al(s))ds + \int_0^t \sigma(X(s),\al(s))dW(s) -Z(t),
\end{displaymath} 
where $W$ is a $d$-dimensional standard Brownian motion, $\al$ is a continuous-time Markov chain with a finite state space $\M=\set{1,\dots, m}$,  $Z=(Z_1,\dots, Z_n)'$ is an $n$-dimensional adapted, nondecreasing, and c\`adl\`ag stochastic process, and $b,\sigma$ are appropriate measurable  functions. 
 The income rates $f_i$, $i=1,\dots,n$, from exerting the singular control are allowed to be state- and regime-dependent.
 The objective is to maximize the total discounted reward
 \begin{equation}\label{eq1-sec1}
\ex \left[\int_0^\infty e^{-rs}\sum_{i=1}^n f_i(X(s),\al(s)) dZ_i(s)\right],
\end{equation}
where $r>0$ is the discounting factor.

Such  singular control problems (in various different settings) have been extensively studied in the literature. A partial list includes 
the monotone follower problems (\cite{Karatzas-S-84}),
  optimal harvesting problems (\cite{A-Shepp,Song-S-Z}),  optimal dividend distribution schemes (\cite{Paulsen-03}),
portfolio selection management with transaction costs (\cite{OS02}),
   optimal partially reversible investment problem (\cite{MR2132595}),  
 and  heavy traffic modeling and control problems (\cite{Wein90,MR2832412}), etc. See also \cite{Hauss-S-I,Hauss-S} for  a general singular stochastic control problem for a multidimensional It\^o diffusion on a fixed time horizon,
in which
the existence of the optimal control and the characterization of the value function as the unique viscosity solution of a Hamilton-Jacobi-Bellman equation are established.  
Singular control problems with state constraints have drawn considerable interests in recent years; see, for example,  \cite{MR2271486, MR2358640,Zar92}, among others. 

Note that most, if not all, of the aforementioned literature on singular stochastic controls deal with It\^o (jump) diffusions. One exception is our recent work \cite{Song-S-Z}, which studies an optimal harvesting problem of a single species living in random environments. 
Due to their capability of modeling complex systems with uncertainty, regime-switching models have drawn considerable attention from both researchers and practitioners in recent decades in a wide range of applications. 
Some of such examples can be found in  mathematical finance (\cite{Zhang}), ecosystem modeling (\cite{Slatkin-78,ZY-09a}),   stochastic manufacturing systems (\cite{SethiZ}), risk management (\cite{Elliott-Siu-10,Zhu-10}), to name just a few. In these systems, both continuous dynamics and discrete events coexist. In particular, the systems often display 
qualitative structural changes.
Regime-switching models turn out to be quite versatile in capturing these inherent randomness. We refer to \cite{MaoY} and \cite{YZ-10} for in-depth investigations of regime-switching diffusions.

	This work aims to investigate the singular control problem \eqref{eq1-sec1} in the setting of multi-dimensional regime-switching diffusion with state constraints.  First we recall the notion of constrained viscosity solution, illustrated by several simple yet nontrivial examples. Then we use the weak dynamic programming principle to show that the value function defined in \eqref{value} is a constrained viscosity solution to the coupled system of quasi-variational inequalities \eqref{qvi-equiv-form} in Theorem \ref{thm-viscosity}.
	Finally, we derive a strong comparison result  in Theorem \ref{thm-comparison}, from which we establish the uniqueness of the constrained viscosity solution to \eqref{qvi-equiv-form}.  
Compared with the classical work on viscosity solution such as  \cite{CIL92,YongZ} and others, the novelty and contribution of this work can be summarized as follows. In lieu of a single differential equation studied in the literature, this work deals with a coupled system of nonlinear second-order differential equations with gradient constraints. The coupling effect is due to the presence of 
random environments or regime switching. This feature at one hand makes our model more appealing in real-world  applications since it can naturally capture the qualitative structural changes of the systems;  on the other hand, it adds much difficulty in the analysis. In particular, the function $F$ defined in \ref{F-defn} is not proper with respect to the variable  $\xi$ in the sense of the User's Guide \cite{CIL92}. Note that the properness was an essential assumption in the proof of strong comparison result in \cite{CIL92}. Here we need to carefully handle the coupling effect; see the proof of Theorem \eqref{thm-comparison} for more details.  
Another  noteworthy feature of this work is that we introduce an exponential transformation which allows us to handle both the gradient constraints as well as the polynomial growth condition on an unbounded domain for the solution of the coupled system of quasi-variational inequalities
 \eqref{qvi-equiv-form}.

The rest of the paper is arranged as follows. Section \ref{sect-formulation} presents the precise formulation of the problem, followed by some preliminary results in Section \ref{sect-DPP}.  We recall the notion of constrained viscosity solution in Section \ref{sect-viscosity}, followed by several examples for illustration.  Further,  in Section \ref{sect-viscosity},  we establish the existence by showing that the value function $V$ defined in \eqref{value} is a constrained viscosity solution of \eqref{qvi-equiv-form}. The strong comparison result is arranged in Section \ref{sect-strong-comparison}. A hierarchical PDE characterization of the boundary behavior is arranged in   Section \ref{sect-strong-comparison} as well.  The paper is concluded with conclusions and remarks in Section \ref{sect-remarks}.

To facilitate later presentation, we introduce some   notations that will be used often in later sections.  We say that a function from $[0,\infty)$ to some Polish space $E$ is c\`adl\`ag if it is right continuous and has
 left limits in $E$ on $[0,\infty)$. When $E=\rr^n $ and $\xi $ is c\`adl\`ag, then we write $\Delta \xi(t)= \xi(t)-\xi(t-)$ for $t> 0$.
 As a convention, we set 
  $\Delta \xi(0)= \xi(0)$.
  Throughout the paper, we use $x'y$ or $x\cdot y$ interchangablly to denote the inner product of vectors $x$ and $y$.
 For any vectors $x,y\in \rr^n$,  $x\le  y$ means $x_i \le y_i$ for every $i=1,2,\dots, n$.
  The space of $n\times n$ symmetric matrices is denoted by $\mathcal S_n$
 and the family of positive definite symmetric matrices is denoted by $\mathcal S_n^+$. If $A,B \in \mathcal S_n$ and $A-B \in \mathcal S_n^+$, then we write
 $A > B$.
   If $\phi: \rr^n \to \rr$ is   sufficiently smooth, 
then    $D_{x_i} \phi= \frac{\partial \phi}{\partial x_i}$, $D_{x_ix_j} \phi= \frac{\partial^2 \phi}{\partial x_i\partial x_j}$,
     and $D\phi   =(D_{x_1}\phi, \dots, D_{x_n}\phi)'$  is the gradient of $\phi$
 while $D^2\phi =(D_{x_ix_j}\phi)$ denotes the Heissian of $\phi$. 
 For any real-valued  function $f$, we use $f_*$ and $f^*$ to denote the lower- and upper-semicontinuous envelopes of $f$, respectively. 
 If $B $ is a set, we use $B^o$ and $I_B$ to denote the interior and indicator function of $B$, respectively. Throughout the paper, we adopt the conventions that $\sup \emptyset =-\infty$
 and $\inf \emptyset = + \infty$. 

\section{Formulation}\label{sect-formulation}
We consider singular control problems for a regime-switching diffusion
 \beq\label{sde-switching} 
 d\zeta  (t)= b(\zeta (t),\al(t)) dt + \sigma(\zeta (t),\al(t))dW(t), \ \ \zeta (0)=x,\  \al(0)=\al,
 \eeq
 where $x\in \rr^n, \al \in \M=\set{1, \dots, m}$,  $W$ is a $d$-dimensional standard Brownian motion, $b: \rr^n \times \M  \mapsto \rr^n$, $\sigma: \rr^n \times \M  \mapsto  \rr^{n\times d}$,
and  $\al\cd \in \M$ is a continuous-time Markov chain that is independent of the Brownian motion $W$ and is generated by
   $Q=(q_{ij})\in \rr^{m \times m}$:
     \begin{equation}\label{Q-gen}\pr\set{\al(t+ \Delta t)=j|
\al(t)=i,\al(s),s\le t}=\begin{cases}q_{ij}
\Delta t + o(\Delta t),\ &\hbox{ if }\ j\not= i\\
1+ q_{ii}\Delta t + o(\Delta t), \ &\hbox{ if }\ j=i,
\end{cases}   \end{equation}  where $q_{ij}\ge 0$ for $i,j=1,\dots,m$
with $j\not= i$ and $ q_{ii}=-\sum_{j\not= i}q_{ij}<0$ for each $i=1,\dots,m$.

Throughout the paper, we assume that the coefficients $b$ and $\sigma$ and the generator $Q$  are such that 
for any initial condition  $(x,\al) \in \rr^n \times \M$, the solution $\zeta ^{\xal}$ to \eqref{sde-switching} exists and is weakly unique.
Sufficient condition for 
existence and uniqueness for stochastic differential equations with regime switching  can be found in, for example, \cite{MaoY,YZ-10}.

We now introduce singular control into \eqref{sde-switching} with state constraint and suppose that the controlled dynamic is given by
 \begin{equation}\label{dyna-vec}
 dX(t) =b(X(t),\al(t)) dt + \sigma(X(t),\al(t))dW(t) -dZ(t),
 \end{equation}
with initial conditions
 \begin{equation}\label{initial}   X(0-) = x \in S,\ \  \al(0)= \al \in \M,  \end{equation}
where $Z \in \rr^n$ is a singular control process to be specified below.  Without loss of generality, we  take $S=\rr^n_+=\set{x\in \rr^n: x_i >0, i=1,\dots, n}$. 
Note that $X(0)$ may not be
 equal to $X(0-)$ due to an instantaneous push $Z(0)$ at time $0$.
Denote the solution to \eqref{dyna-vec} with initial condition specified by (\ref{initial}) by $  X^{x,\al}\cd$.

   Let $\mathcal A_{\xal}$ denote the collection of all {\em admissible controls} with initial conditions given by \eqref{initial}, where  $Z\in \mathcal A_{x,\al}$ satisfies
 \begin{itemize}
 \item[(i)] for each $i=1,\dots,n$, $Z_i(t)$ is nonnegative,  c\`adl\`ag  and nondecreasing with respect to $t$,
  \item[(ii)]  $X(t) \in S$ for all $t\ge 0$, and

 \item[(iii)] $Z(t)$ is adapted to ${\F}_t:=\sigma\set{W(s),\al(s), 0\le s \le t}$, where $\F_0$ contains all $\pr$-null sets. Moreover,
     $$\ex \int_0^\infty e^{-rs} d\abs{Z}(s)< \infty.$$
 \end{itemize}
	Note that the state constraint is specified in  condition (ii) above.  
Throughout the paper, we assume	 $\mathcal A_{\xal}\not= \emptyset$ for every $(x,\al) \in  S \times \M$; see Section \ref{sect-DPP} for a sufficient condition. 
  For a fixed $Z\in \A_{\xal}$, the discounted payoff is 
  \beq\label{J-defn}J(x,\al,Z) :=\ex \int_0^{\infty} e^{-r s}   f( X^{x,\al}(s-),\al(s-)) \cdot dZ(s),\eeq
where $f: S \times \M \mapsto  {\rr^n}$ with  $f_i$ representing the state- and regime-dependent  instantaneous marginal
yields accrued from exerting the singular control $Z_i(t)$. Assume $f_i$ is
continuous and non-increasing with respect to $x$ in the sense that
$f_i(x,\al) \ge f_i(y,\al)$ for each $\al\in\M$ if $x\le   y$,
where $x=(x_1,\dots,x_n)'$ and $y=(y_1,\dots,y_n)'$ satisfy $x_j \le y_j$ for each $ j=1,\dots,n$. Moreover, we assume $0< f_i(0,\al) < \infty$ for
each $i=1,\dots, n$ and $\al \in \M$.  Such assumptions on $f$ are motivated by considerations in optimal harvesting problems  (\cite{Alvarez} and  \cite{Song-S-Z}).
 The goal is to maximize the expected total discounted payoff  and find an optimal control $Z^*$:
\begin{equation}\label{value}
V(x,\al) = J(x,\al,Z^*) := \sup_{Z\in \mathcal{A}_{\xal}} J(x,\al,Z).
\end{equation}
In order to work with a well-formulated maximization problem, we assume throughout the paper that $V(x,\al)< \infty$ for all $(x,\al)\in \rr_{+}^{n} \times \M$.



As usual, we shall rely on the   dynamic  programming principle  (DPP) 
to deduce the behavior of the value function \begin{equation}\label{dyna-prog}\begin{aligned}
V(x,\al)=  \sup_{Z\cd\in \Ax} \ex  \bigg[ \int_0^{  \eta} e^{-rs}   f(X^{x,\al}(s-),\al(s-))\cdot dZ(s)
 +    e^{-r  \eta}V(X^{x,\al}( \eta),\al( \eta))\bigg]
\end{aligned}\end{equation}
for every $(x,\al)\in S\times\M$ and stopping time $\eta$. 
A heuristic argument using the DPP \eqref{dyna-prog} yields that $V$ satisfies the following coupled system of quasi-variational inequalities
\begin{equation}\label{qvis}
\min\set{(r-\op) V(x,\al), \min_{i=1,\dots, n}\set{ D_{x_i}V(x,\al)-f_i(x,\al)}} = 0,\ \  (x,\al) \in S \times \M,
\end{equation}
where  for any $h(\cdot, \al)\in C^2$, $\al \in\M$, we define
\beq\label{op-defn}\barray
\op h(x,\al)\ad =  \frac{1}{2}\tr(\sigma \sigma'(x,\al)D^2 h(x,\al))  +  b(x,\al) \cdot D h(x,\al)
 + \sum_{j=1}^m q_{\al j}h(x,j).
   \earray
\eeq

However, without a priori result on the continuity of the value function, a rigorous proof of \eqref{dyna-prog} 
is nontrivial. Thanks to the state constraint as well as the generality of the set up of the problem,  
it seems not easy to obtain the continuity of the value function $V$ defined in \eqref{value}. Also, in the current singular control setup with regime-switching diffusion, it appears that 
the  DPP is not available from the literature.  
To overcome this difficulty, we will instead invoke the weak DPP (\cite{Bouchard-Touzi-11}); see Section \ref{sect-DPP} for the  precise statement. 
Also, the value function $V$ is not necessarily sufficiently smooth to take first and second order partial derivatives. Therefore  we aim to show in this work that $V$ satisfies \eqref{qvis} in the weak sense using the notion of   viscosity solution.  We will show that the value function $V$ is the unique viscosity solution to \eqref{qvis}. 

For convenience of later presentations, we define for any $(x,\al,\xi, p, A)\in \rr^n \times \M \times \rr^m \times \rr^n \times \SS_n$,
\beq\label{F-defn}
F_\al(x,\xi, p,A) =F(x,\al,\xi, p,A): = r \xi_\al - \frac{1}{2}\tr(\sigma\sigma'(x,\al) A) -b(x,\al) \cdot p -\sum_{j=1}^m q_{\al j} \xi_j.
\eeq
Set $\V(x)=(V(x,1), \dots, V(x,m))' \in \rr^m$, then \eqref{qvis} can be rewritten  as
\beq\label{qvi-equiv-form}
\min\set{F_\al(x,\V(x), DV(x,\al), D^2 V(x,\al)), \min_{i=1,\dots, n}\set{ D_{x_i} V(x,\al) -f_i(x,\al)}} =0, \eeq
for all $x\in S$ and each $ \al=1,\dots, m$.

As we indicated earlier,
  \eqref{qvi-equiv-form} is a coupled system of quasi-variational inequalities. Moreover, thanks to the term $\sum_{j=1}^m q_{\al j} \xi_j$ with $Q=(q_{ij})$ defined in \eqref{Q-gen},  for each $\al \in \M$, $F_\al$  is not proper with respect to the variable $\xi$   in the sense of  equations (0.1) or (0.2) in the User's Guide \cite{CIL92}.
   Note that properness assumption (and in particular equation (3.13) in \cite{CIL92})  enabled them to derive the strong comparison result
   and hence the uniqueness  of the viscosity solution. Here for our analysis, special care has to be given to handle the fact that $F_\al$ is not proper due to the coupling term. 
   Also, instead of working on a bounded domain, we are dealing with  unbounded domain $S$.  
These features make  our analysis much more involved than the classical
comparison result in  \cite{CIL92}.

\section{Some Preliminary Results}\label{sect-DPP}

We present some preliminary results in this section. The first one provides a sufficient condition for the assumption that $\Ax \not=\emptyset$ for all $(x,\al) \in S\times \M$. 


\begin{prop}
Assume there exists a function $\Psi: S\times \M\mapsto \rr_+$ satisfying 
\begin{itemize} 
\item[(i)] for each $i=1,\dots, n$ and  $\al \in \M$, \begin{displaymath}
\lim_{|x|\to \infty} \Psi(x,\al)=\infty, \text{ and }\lim_{x_i\downarrow 0} \Psi(x,\al) = \infty, 
\end{displaymath} 
\item[(ii)] $\Psi(\cdot, \al)\in C^2$ for each $\al \in \M$ and   $$ \op \Psi(x,\al) = b(x,\al)\cdot D\Psi(x,\al) + \frac{1}{2}\tr(\sigma\sigma'(x,\al) D^2 \Psi(x,\al)) + \sum_{j=1}^m q_{\al j} \Psi(x,j)\le 0,$$ 
\end{itemize}    for all  $ (x,\al) \in S\times \M$.
Then, denoting by $\zeta^{x,\al}$ the solution of \eqref{sde-switching} with initial condition $(x,\al)$,  we have \begin{equation}
 \label{eq-constraint}
 \pr\set{\zeta ^{x,\al}(t)\in S \text{ for all }t \ge 0}=1, \ \ \text{ for any }(x,\al) \in S\times \M.
 \end{equation} 
  Consequently $\Ax\not=\emptyset$ for all $(x,\al)\in S\times \M$.
\end{prop}
\begin{proof}
It suffices to prove 
\eqref{eq-constraint}, which leads to   $Z\equiv 0 \in \Ax $ for any  $(x,\al) \in S\times \M$. To this end, we consider $(x,\al)\in S\times \M$ and define
$$ \tau_k:=\inf\set{t\ge 0: \abs{\zeta^{x,\al}(t)} \ge k \text{ or } \zeta_i^{x,\al}(t) \le \frac{1}{k} \text{ for some }i=1,\dots,n}.$$ Note that $\{\tau_k\}$ is a nondecreasing sequence of $\F_t$-stopping times. 
Now it is enough to show that $ \tau_k \to \infty $ with probability 1.  Suppose on the contrary that \begin{equation}
\label{eq1-prop31}
\liminf_{k\to \infty}\pr\set{\tau_k < \infty} = \delta >0.
\end{equation} Applying generalized It\^o's formula to the function $\Psi$ and using condition (ii), we obtain that for any $t\ge 0$, 
\begin{displaymath}
\begin{aligned}
\ex\left[\Psi(\zeta(\tau_k \wedge t), \al(\tau_k \wedge t)\right] & = \Psi(x,\al) + \ex\left[\int_0^{\tau_k \wedge t} \op \Psi(\zeta(s),\al(s))ds \right] \\
 & \le \Psi(x,\al).
 \end{aligned}
\end{displaymath}
Then since $\Psi \ge 0$,  it follows   from condition (i) and \eqref{eq1-prop31} that 
\begin{displaymath}\begin{aligned}
\Psi(x,\al) &  \ge \ex\left[\Psi(\zeta(\tau_k \wedge t), \al(\tau_k \wedge t)\right] \ge \ex\left[\Psi(\zeta(\tau_k),\al(\tau_k)I_{\set{\tau_k \le t}}\right] \\ 
& \ge  \Psi_k\pr\set{\tau_k \le t} \to \infty, \text{ as } k \to \infty,
 \end{aligned}
\end{displaymath} 
where $$\Psi_k : =\inf\set{\Psi(x,j), |x|=k \text{ or } x_i = \frac{1}{k} \text{ for some }i=1,\dots, n, \text{ and }j\in \M} .$$
This is a contradiction and hence $\tau_k \to \infty $ with probability 1 as $k \to \infty$.
\end{proof}

We will need the following proposition in the proof of Theorem \ref{thm-viscosity}.
\begin{prop}\label{lem-increasing}
  For each $\al \in \M$ and any $x,y \in S$ with $y\le x$, we have
  \begin{align}\label{value-increasing}V(x,\al) & \ge   f(x,\al)\cdot ( x-y) + V(y,\al),\\ 
  \label{Vstar} V^*(x,\al) & \ge   f(x,\al)\cdot ( x-y) + V^*(y,\al).\end{align}
\end{prop}
\begin{proof}Equation \eqref{value-increasing} can be established using exactly the same arguments as those in \cite{Song-S-Z}, while \eqref{Vstar} follows from \eqref{value-increasing} directly.\end{proof}

The next proposition can be established using similar arguments as those in \cite{Bouchard-Touzi-11}.
\begin{prop}\label{prop-DPP}
Fix $(x,\al) \in \rr_+^n \times \M$. Then for any stopping time $\tau$,  we have 
\begin{equation}
\label{DPP-eq1}
V(x,\al) \le \sup_{Z\in \Ax} \ex \left[ \int_0^\tau e^{-rs} f(X^{x,\al}(s-),\al(s-) \cdot dZ(s) + e^{-r \tau} V^*(X^{x,\al}(\tau), \al(\tau))\right],
\end{equation} 
and 
\begin{equation}
\label{DPP-eq2}
V(x,\al) \ge \sup_{Z\in \Ax} \ex \left[ \int_0^\tau e^{-rs} f(X^{x,\al}(s-),\al(s-) \cdot dZ(s) + e^{-r \tau} \varphi(X^{x,\al}(\tau), \al(\tau))\right],
\end{equation} for all upper-semicontinuous functions $\varphi$ such that $V\ge \vphi$ on $\rr_+^n \times \M$. 
\end{prop}

We finish this section with the   verification theorem, whose proof is similar to those in \cite{Song-S-Z}. 

\begin{thm}\label{thm-verification}
 Suppose there exists a function $\phi: S \times \M \mapsto \rr_+$ solving \eqref{qvis}.

\begin{itemize}\item[{\em (a)}] Then $\phi(x,\al) \ge V(x,\al)$ for  every $(x,\al)\in S \times \M$.

\item[{\em (b)}] Define the {\em non-intervention region}
$$ {\cal C}= \set{(x,\al)\in S\times \M:
\max_{i=1,\dots,n}\set{ f_i(x,\al)-D_{x_i}\phi(x,\al)} < 0}. $$
Assume  
 there exists a   strategy $ Z \in \mathcal A_{x,\al}$ such that
 \begin{align}\label{cond-op0} & ( X(t),\al(t)) \in  {\cal C} \text { for Lebesgue almost all }  0 \le t  < \infty,   \\
 \label{cond-op1}
&\int_0^t  ( D \phi( X(s),\al(s)) -f(X(s),\al(s))) \cdot  d \wdt Z^c(s)   =0, \text{ for any } 0\le  t  < \infty,\\
\label{cond-op3}&  \lim_{N\to \infty} \ex  \left[e^{-r(\tau\wedge N \wedge \beta_N)} \phi(X(\tau\wedge N \wedge \beta_N),\al(\tau\wedge N \wedge \beta_N))\right]=0,
\end{align}
and that if $  X(s) \not=  X(s-)$, then
\begin{equation}\label{cond-op2}
  \phi( X(s),\al(s-))-\phi(X(s-),\al(s-)) = -   f(\wdt X(s-),\al(s-)) \cdot  \Delta \wdt Z(s) ,
\end{equation}
where $\beta_N:=\inf\{t\ge 0: | X(t)| \ge N\},$ and  $ X =X^{x,\al}$ denotes the solution of \eqref{dyna-vec}. Then $\phi(x,\al)= V(x,\al)$ for every $(x,\al )\in S \times \M$ and $\wdt Z $ is an optimal   strategy.
\end{itemize} \end{thm}

\section{Viscosity Solution: Existence}\label{sect-viscosity}
This section is devoted to the properties of the value function $V$.
In particular, we aim to characterize $V$ as a viscosity solution to the quasi variational inequality \eqref{qvi-equiv-form}. Let's first recall the notion of viscosity solution.  

\begin{defn}\label{d-visc}    A function
${\mathbf u}(x)= (u(x,1), \dots, u(x,m))'$ is said to be
a {\em viscosity subsolution}   of
\eqref{qvi-equiv-form} on $\bar S \times \M$, if
 for any $(x_0,\al_0) \in \bar S \times \M$ and functions
  $\varphi(\cdot, \al) \in C^2(S), \al \in \M$ satisfying $(u^*-\vphi)(\xal)  \le (u^*-\vphi)(\xalz)  =0  $   for
  all $(x,\al) \in \bar S\times \M$, we have
  $$\min\set{F_{\al_0}(\xz,{\mathbf u}^*(\xz), D\vphi(\xalz), D^2 \vphi(\xalz)), \min_{i=1,\dots, n}\set{ D_{x_i} \vphi(\xalz) -f_i(\xalz)}}  \le 0. $$
 Similarly, a function
${\mathbf u}(x)= (u(x,1), \dots, u(x,m))'$ is said to be
a {\em viscosity supersolution}   of
\eqref{qvi-equiv-form} in $S\times \M$, if
 for any $(x_0,\al_0) \in S \times \M$ and functions
  $\varphi(\cdot, \al) \in C^2(S),\al\in \M$ satisfying $(u_*-\vphi)(\xal)  \ge (u_*-\vphi)(\xalz)  =0  $ for
  all $(x,\al) \in S\times \M$, we have
  $$\min\set{F_{\al_0}(\xz,{\mathbf u}_{*}(\xz), D\vphi(\xalz), D^2 \vphi(\xalz)), \min_{i=1,\dots, n}\set{ D_{x_i} \vphi(\xalz) -f_i(\xalz)}}  \ge 0. $$
The function
$u$ is said to be a {\em constrained viscosity solution}, if it
is both a viscosity
subsolution in $\bar S \times \M$ and a viscosity supersolution in $S\times \M$.
\end{defn}

Before presenting the main result of this section, we shall first study several examples to   illustrate Definition  \ref{d-visc}. These examples will also help us to  motivate later results.
\begin{exm}\label{exm1}
Consider the QVI
\begin{equation}
\label{qvi-ex1}
\min\set{u(x)-u'(x), u'(x)-1}=0, \ \ x\in (0,\infty).
\end{equation}
We claim that $u(x)= K e^x$, $K\ge 1$ is a constrained viscosity solution of \eqref{qvi-ex1} on $[0,\infty)$. In fact, if $x>0$, then we compute
\begin{displaymath}
 {\min\set{u(x)-u'(x), u'(x)-1}} = \min\set{0, Ke^x -1}=0.
 \end{displaymath}
Therefore it remains to verify that $u(x)=Ke^x$ is a subsolution on $[0,\infty)$ using Definition \ref{d-visc}.  Suppose $\phi\in C^1$ and satisfies 
$(u-\phi)(x) \le (u-\phi)(0)=0$ for $x\in [0,\infty)$ in a neighborhood of $0$. Then it follows that $\phi'(0) \ge u'(0) =K$ and hence 
\begin{displaymath}
\begin{aligned}
\min\set{\phi(0)- \phi'(0), \phi'(0)- 1} = \min\set{K -\phi'(0), \phi'(0)-1} =K -\phi'(0) \le 0. 
\end{aligned}
\end{displaymath} 
Thus the claim follows. 

Next we show  that $v(x)=x+1$ is also a constrained viscosity solution on $[0,\infty)$. In fact, it is easy to see that $v(x)=x+1 $ solves \eqref{qvi-ex1} for $ x>0$. Thus it remains to show 
that it is also a subsolution on $[0,\infty)$. To this end, let $\vphi \in C^1$ with $(v-\vphi)(x) \le (v-\vphi)(0) =0$ for $x\in [0,\infty)$ in a neighborhood of $0$. Then we have $\vphi'(0) \ge v'(0) =1$ and therefore   
 \begin{displaymath}
\min\set{\vphi(0)-\vphi'(0),\vphi'(0)-1}= \min\set{1-\vphi'(0),\vphi'(0)-1} = 1- \vphi'(0) \le 0 .
\end{displaymath}
This shows that $v$ is a subsolution and thus a constrained solution on $[0,\infty)$.

Note that the  controlled process corresponding to \eqref{qvi-ex1} is $dX(t)= 1\cdot dt + 0 \cdot dW(t)-dZ(t)$ or $X(t) = x+ t-Z(t)$ for $t\ge 0$ and the objective is to maximize $J(x,Z)=\ex_x\int_0^\infty e^{-t}dZ(t)$. For this process, it is clear that $\mathcal A_x \not= \emptyset$ for all $x\in[0,\infty)$. 
Moreover, from the state constraint, $Z(t) \le x+t$ for all $t\ge 0$. Then it follows that 
\begin{displaymath}
J(x,Z) =   \ex_x\int_0^\infty \int_t^\infty e^{-s}ds dZ(t) = \ex_x \int_0^\infty \int_0^s dZ(t) e^{-s} ds  \le \ex_x \int_0^\infty e^{-s} (x+s)ds = x+1. 
\end{displaymath} Thus the value function $V(x) \le x+1$. 
In fact,  $V(x)=x+1$ and $Z^*(t) := xI_{\set{t=0}} +   I_{\set{t>0}}\int_0^t s ds$ is an optimal control, since $J(x,Z^*) = x+ \int_0^\infty t e^{-t}dt =x+1 $. 

To conclude, the value function $V$ is the unique constrained viscosity solution of \eqref{qvi-ex1} on $[0,\infty)$ in  the class of functions with  polynomial growth rate. 
\hfill $\Box$
\end{exm}

\begin{exm}\label{exm2}
In this example, we demonstrate that the QVI
\begin{equation}
\label{qvi-ex2}
\min\set{u(x)-u''(x), u'(x) -1 }=0, \ \ x\in(0,\infty)
\end{equation}
has no constrained viscosity solution on $[0,\infty)$. 

First, one can show that $u(x)=x + c$ is not a constrained  viscosity solution of \eqref{qvi-ex2} on $[0,\infty)$, where $c$ is a constant. Certainly it is the case if $c<0$ since for $x\in (0, -c)$, we have $\min \set{u(x)-u''(x), u'(x)-1} =\min\set{x+ c, 1-1} =x+c < 0$. 
Now let's consider the case when $c \ge 0$.  
The function $\phi(x)= c+ 2x-(1-\frac{c}{2}) x^2$ satisfies 
\begin{displaymath}
(u-\phi)(x) = -x + (1-\frac{c}{2}) x^2  \le (u-\phi)(0) =0, \ \ \text{ for }x\ge 0 \text{ sufficiently small},
\end{displaymath} 
and $\phi'(x)=2-(2-c)x$, $\phi''(x)=-2+ c$. Thus we have 
\begin{displaymath}
\min\set{\phi(0)-\phi''(0), \phi'(0)-1}= \min\set{c-(-2+ c), 2-1} >0;
\end{displaymath}
this shows that $u(x)=x+c$ is not a subsolution on $[0,\infty)$.

Next we show that $u(x)= c_1 e^x + c_2 e^{-x}$ is not a constrained viscosity solution of \eqref{qvi-ex2} on $[0,\infty)$ either, where $c_1, c_2$ are constants.
Note that for $x\ge 0$ small, $$ \begin{aligned}
u(x)& = c_1 (1+ x + \frac{1}{2}x^2 + o (x^2) )+ c_2(1- x+ \frac{1}{2}x^2 + o (x^2) ) \\ &= (c_1 +c_2) + (c_1-c_2)x + \frac{1}{2}(c_1+ c_2) x^2 + o(x^2). \end{aligned}$$
If $c_1 + c_2>0$, then we consider $\phi(x)=(c_1 +c_2)  + (\abs{c_1-c_2} + 2)x + \frac{1}{3}(c_1+c_2)x^2$. Clearly we have $(u-\phi)(x) \le (u-\phi)(0) =0$ for $x$ small and 
\begin{displaymath}
\min\set{\phi(0)-\phi''(0), \phi'(0)-1}= \min\set{(c_1+c_2) -\frac{2}{3}(c_1+c_2), \abs{c_1-c_2}+2-1} > 0.
\end{displaymath} Thus $u$ is not a constrained viscosity solution on $[0,\infty)$. 

Now we consider the case when $c_1 + c_2 \le 0$. Let $\phi(x) =  (c_1 +c_2)  + (\abs{c_1-c_2} + 2)x + (c_1+c_2-1)x^2$. Then we can verify  $(u-\phi)(x) \le (u-\phi)(0) =0$ for $x$ small and  $\phi'(x) = \abs{c_1-c_2} + 2 + 2(c_1+c_2-1) x$  and $\phi''(x) = 2(c_1+c_2-1) $. Then we compute 
\begin{displaymath}
\begin{aligned}
\min\set{\phi(0)-\phi''(0), \phi'(0)-1} & = \min\set{c_1+ c_2 - 2(c_1+c_2) + 2, \abs{c_1-c_2} + 2 -1} \\
  & = \min\set{2-(c_1+c_2), \abs{c_1-c_2} +  1} >0;
\end{aligned}\end{displaymath}
which again demonstrates that $u$ is not a constrained viscosity solution of \eqref{qvi-ex2} on $[0,\infty)$.  

One observes that any linear combination of $x+c$ and $c_1 e^x + c_2 e^{-x}$ can not be a constrained viscosity solution of \eqref{qvi-ex2} on $[0,\infty)$ either. In addition, functions of the form $u(x)= (x+c) I_{\set{x>a}} + (c_1 e^x + c_2 e^{-x})I_{\set{x\le a}}$ are not   constrained viscosity solution of \eqref{qvi-ex2} on $[0,\infty)$, where $a, c_1, c_2$ are appropriately selected  constants so that $u\in C^1([0,\infty))\cap C^2([0,\infty)-\set{a})$ and solves \eqref{qvi-ex2} in $(0,\infty)$. 

Finally we note that for the corresponding controlled process 
$X(t) =x+ \sqrt{ 2}W(t) $ and the reward functional $\ex_x\int_0^\infty e^{-t} dZ(t)$, $\mathcal A_0 =\emptyset$. The reason is that the Brownian motion $W$, starting from $0$, changes sign infinitely many times  and hence can not satisfy the state constraint in any time interval $[0,\e]$. \hfill $\Box$
\end{exm}

\begin{exm}\label{exm3}
In this example, we consider the system of coupled QVIs 
\begin{equation}
\label{qvi-ex3}
\begin{aligned}
\min\Big\{ru(x,\al)-\mu_\al x u'(x,\al) - \frac{1}{2}\sigma_\al^2 x^2 u''(x,\al)- \lambda_\al u(x,\al) + \lambda_\al u(x,3-\al), \\  u'(x,\al) - 1\Big\}=0, \ \ x\in (0,\infty), \ \al \in \set{1,2},
\end{aligned}\end{equation}
where for $\al=1,2$,   $\mu_\al, \sigma_\al$, and $ \lambda_\al >0$  are constants. Moreover, we assume $\mu_1,\mu_2$ satisfy $\mu_1  < r < \mu_2 \le    \frac{r\lambda_1 + (r-\mu_1)(r+ \lambda_2)}{r+ \lambda_1 -\mu_1 }$. 
 One can verify that the unique solution  to \eqref{qvi-ex3} in $(0,\infty)\times \M$
is \begin{equation}\label{soln-1d-ex} u(x,1)=x,\ \ \ u(x,2)=\frac{\lambda_2}{\lambda_2+r-\mu_2}x, \ \ \ x> 0. \end{equation}  
Moreover, one can easily verify that  $u(\cdot,\al), \al=1,2$ satisfy the subsolution property at the point $x=0$. Therefore $u$ is the unique constrained solution on $[0,\infty)\times \set{1,2}$.   

The corresponding controlled dynamic is given by the regime-switching geometric Brownian motion: 
\begin{displaymath}
dX(t) = \mu_{\al(t)} X(t) dt + \sigma_{\al(t)} X(t)dW(t) -dZ(t),
\end{displaymath} 
where $\set{\al(t), t\ge 0}$ is a two-state continuous-time Markov chain with generator $\begin{pmatrix}-\lambda_1 & \lambda_1 \\ \lambda_2 & -\lambda_2 \end{pmatrix}$. The objective is maximize the reward $J(x,\al,Z) = \ex_{x,\al} \int_0^\infty e^{-rt} dZ(t)$. Observe that $\mathcal A_{x,\al} \not= \emptyset$ for all $(x,\al) \in  [0,\infty) \times \set{1,2}$. Moreover, as demonstrated in \cite{Song-S-Z}, the value function $V(x,\al) = u(x,\al)$ for all $(x,\al) \in [0,\infty) \times \set{1,2}$, where $u$ is defined in \eqref{soln-1d-ex}. 
 \hfill $\Box$\end{exm}

Now let's present the main result of this section.

\begin{thm}\label{thm-viscosity}
	Assume  
$\mathcal A_{x,\al}\not=\emptyset$	and  that the value function $V(\cdot, \al)$ is  finite for each $(x,\al)\in \bar S \times \M$. 
	Then  $\V(x)=(V(x,1),\dots, V(x,m))'$ is a constrained viscosity solution of \eqref{qvi-equiv-form} on $\bar S \times \M$.
 \end{thm}

 The proof of Theorem \ref{thm-viscosity} is accomplished by the combination of Propositions \ref{prop-supersoln} and \ref{prop-subsoln}:  Proposition  \ref{prop-supersoln}  shows that $\V$ is a viscosity supersolution, while Proposition  \ref{prop-subsoln}   establishes that $\V$ is viscosity subsolution.

 \begin{prop}\label{prop-supersoln}
 The function $\V$ is a viscosity supersolution of \eqref{qvi-equiv-form} in $S\times \M$. That is, for any $(x_0,\al_0) \in S\times \M$ and any $C^2$ function $\phi(\cdot, \cdot)$ satisfying $\phi(x_0,\al_0)= V_*(x_0,\al_0)$ and that $\phi(x,\al) \le V_*(x,\al)$ for all $x$ in a neighborhood of $x_0$ and each $\al\in \M$, we have
 \begin{equation}\label{viscosity-sub}
 \min\set{(r-\op) \phi(\xalz), \min_{i=1,\dots, n} \set{D_{x_i} \phi(x_0,\al_0)-f_i(\xalz)}} \le 0.\end{equation}
\end{prop}
  
  \begin{proof}
   By the definition of $V_*(\xalz)$, there exists a sequence $\set{x_m} \subset \rr_+^n$ such that 
 \begin{equation}\label{eq-xm-x0}
x_m \to x_0, \text{ and } V(x_m,\al_0) \to V_*(\xalz), \text{ as }m \to \infty.
\end{equation}
This, together with the continuity of $\phi$, implies that  
 \begin{displaymath}
\gamma_m:= V(x_m,\al_0) - \phi(x_m,\al_0) \to 0, \text{ as } m \to \infty.
\end{displaymath}
 Let $B_\e(x_0):=\set{x\in \rr^n: \abs{x-x_0} < \e}$, where $\e>0$ is sufficiently small so that (i) $\lbar B_\e(x_0) \subset S$ and (ii) $\phi(x,\al)
\le V_*(x,\al)$ for all $(x,\al)\in \lbar B_\e(x_0)\times \M$, where $\lbar B_\e(x_0)= \set{x\in \rr^n: \abs{x}\le \e}$ denotes the closure of $ B_\e(x_0)$.
Choose $Z$ such that $Z(0-)=0$ and $Z(t)=\eta $ for all $t \ge 0$, where $0\le |\eta| < \e/2$. Then thanks to \eqref{eq-xm-x0}, $Z\in \mathcal A_{x_m,\al_0}$ for $m$ sufficiently large. 
Let $ X\cd=X^{x_m,\al_0}(\cdot; Z)$ be the corresponding controlled  process with initial condition $(x_m,\al_0)$ and control strategy $Z\cd$.
Put \bed \theta_m:= \inf\set{t \ge 0:  X(t) \notin B_{\e }(x_0)}.\eed
Let $\{h_m\}$ be a strictly positive sequence such that 
\begin{displaymath}
h_m\to 0 \text{ and } \frac{\gamma_m}{h_m} \to 0 \text{ as } m \to \infty. 
\end{displaymath}
Note that the chosen control strategy $Z$ guarantees that $ X\cd$ has at most one jump at $t=0$ and remains continuous on $(0,\theta_m]$.
This, together with  the choice of $\e$, implies that 
$ X(t) \in \lbar B_\e(x_0)$ for all $0\le t \le \theta_m$. 
Since $\phi \le V_* \le V$, we can apply the dynamic programming principle \eqref{DPP-eq2} to obtain 
\begin{equation}
\label{DPP-V-xm} 
\begin{aligned}
V(x_m,\al_0) \ge &\  \ex\left[\int_0^{\theta_m\wedge h_m}\!\! e^{-rs} f(X(s-),\al(s-))\cdot dZ(s) \right] \\ &  + \ex\left[e^{-r(\theta_m \wedge h_m)} \phi(X(\theta_m\wedge h_m), \al(\theta_m\wedge h_m))\right].
\end{aligned}
\end{equation}
On the other hand, It\^{o}'s formula yields 
\begin{equation}
\label{super-ito}
\begin{aligned}
\phi(x_m,\al_0)  =\  & \ex \left[ e^{-r(\theta_m \wedge h_m)} \phi(X(\theta_m\wedge h_m), \al(\theta_m\wedge h_m)) \right] \\ 
 & + \ex\left[ \int_0^{\theta_m \wedge h_m} e^{-rs}(r-\op) \phi(X(s),\al(s))ds\right] \\  &  + \ex\left[  \int_0^{\theta_m \wedge h_m}  e^{-rs} D\phi(X(s),\al(s))\cdot dZ^c(s) \right]\\ 
 & - \ex\left[\sum_{0\le s \le \theta_m \wedge h_m} e^{-rs}[\phi(X(s),\al(s-))-\phi(X(s-),\al(s-))]\right],
\end{aligned}\end{equation}
where in the above, we have used the fact that 
$$ \ex\left[\int_0^{\theta_m\wedge h_m} e^{-rs}  D\phi(X(s),\al(s)) \cdot \sigma(X(s),\al(s))dW(s)\right] =0.$$
A combination of \eqref{DPP-V-xm} and \eqref{super-ito} yields 
\begin{equation}\label{sub-ineq3}
\begin{aligned}
\gamma_m = &\  V(x_m,\al_0) - \phi(x_m,\al_0) \\
               \ge &\  \ex\left[\int_0^{\theta_m\wedge h_m}\!\! e^{-rs} \left(f(X(s-),\al(s-))\cdot dZ(s) -  
                (r-\op) \phi(X(s),\al(s))ds\right)\right]
               \\ & -    \ex\left[  \int_0^{\theta_m \wedge h_m}  e^{-rs} D\phi(X(s),\al(s))\cdot dZ^c(s) \right] \\ 
                & + \ex\left[\sum_{0\le s \le \theta_m \wedge h_m} e^{-rs}[\phi(X(s),\al(s-))-\phi(X(s-),\al(s-))]\right].
\end{aligned}\end{equation}
 Now let $\eta=0$, i.e., $Z(t)\equiv 0$ for any $t\ge 0$. Then \eqref{sub-ineq3} can be rewritten as
  \bed  \begin{aligned} \frac{\gamma_m}{h_m}& \ge - \frac{1}{h_m} \ex \int_0^{\theta_m\wedge h_m} e^{-r s} (r- \op) \phi( X(s),\al(s))ds
    \\ & = - \frac{1}{h_m}   \ex \int_0^{ h_m} e^{-r s} (r-\op) \phi( X(s),\al(s))I_{\set{s \le \theta_m}}ds.
   \end{aligned}\eed
      Note that as $m\to \infty$ and hence $h_m \to 0$,  by the right continuity of the trajectory $(X(s),\al(s))$, $e^{-r s} (r- \op) \phi( X(s),\al(s))I_{\set{s \le \theta_m}} \to (r- \op) \phi(\xalz)$ a.s.  for $s\in [0,h_m]$. 
  Thus by virtue of the bounded convergence theorem, we have
  \beq\label{sub-ineq4} (r-\op) \phi(\xalz) \ge 0.\eeq

  On the other hand, if we choose $\eta=\eta_i e_i $ with $0< \eta_i< \e/2$ and
   $e_i=(0,\dots,1, \dots, 0)'$ being the $i$th unit vector,
   $i=1,\dots, n$, then (\ref{sub-ineq3})
  reduces to
  \bed \ga_m \ge  - \ex \int_0^{\theta_m\wedge h_m} e^{-r s} (r- \op) \phi(X(s),\al(s))ds  + f_i(x_m,\al_0) \eta_i + \phi(x_m-\eta,\al_0) -\phi(x_m,\al_0) .\eed
  Now sending $m\to \infty$, we have \bed f_i(\xalz) \eta_i + \phi(x_0-\eta,\al_0) -\phi(\xalz) \le 0.\eed
  Finally, dividing the above inequality by $\eta_i$ and letting $\eta_i \to 0$ lead to 
  \beq\label{sub-ineq5} D_{x_i}\phi(\xalz)- f_i(\xalz)   \ge 0, \ \ i=1,\dots, n. \eeq
  Now (\ref{viscosity-sub}) follows from a combination of (\ref{sub-ineq4}) and (\ref{sub-ineq5}).
\end{proof}

\begin{prop}\label{prop-subsoln} 
The function $\V$ is  a viscosity subsolution of \eqref{qvi-equiv-form} in $\bar S \times \M$.
   That is, for any $(\xalz) \in \bar S \times \M $ and any $\varphi \in C^2$ such that $\varphi(\xalz)=V^*(\xalz)$ and that $\varphi(\xal)\ge V^*(\xal)$ for $x\in\bar S$ in a neighborhood of $x_0$ and each $\al\in \M$,
   we have
   \begin{equation}\label{viscosity-super}
   \min\set{(r-\op) \varphi(\xalz),\min_{i=1,\dots,n}\set{D_{x_i}\varphi(\xalz)-f_i(\xalz)}} \le 0.
   \end{equation}
  \end{prop}
  
  \begin{proof}
 Suppose on the contrary that (\ref{viscosity-super})   was wrong, then there would exist some $(\xalz)\in \bar S\times \M$, a $\varphi\in C^2$ with $(V^*-\vphi)(x,\al) \le (V^*-\vphi)(x_0,\al_0) =0$, and a constant $A>0$ such that
 \begin{equation}\label{vis-super-contra}
\min\set{(r-\op) \varphi(\xalz),\min_{i=1,\dots,n}\set{D_{x_i}\varphi(\xalz)-f_i(\xalz)}}  \ge 2A > 0.
 \end{equation}

 In what follows, we will derive a contradiction to \eqref{vis-super-contra}. This is achieved in several steps. First we use the generalized It\^o formula and \eqref{vis-supersoln-contra} to obtain \eqref{eq-vphi-new}, 
 from which we obtain \eqref{Vm-estimate-eq1} and \eqref{super-ineq1}. Next, detailed analysis using the monotonicity of the functions $V^*$ and $f$ leads
 to (\ref{super-ineq5}). Then we claim in (\ref{super-claim2}) that the last term in (\ref{super-ineq5}) is bounded below by a positive constant, from which, with
 the aid of dynamic programming \eqref{DPP-eq1}, we obtain a contradiction to (\ref{vis-super-contra}).  The final step of the proof is devoted to the proof of (\ref{super-claim2}).
 
{\em Step 1.}  As in the proof of Proposition \ref{prop-supersoln}, let $\set{x_m} \subset \bar S$ be a sequence such that 
 $$x_m \to x_0, \text{ and } V(x_m,\al_0) \to V^*(x_0,\al_0)Ê\text{ as } m \to \infty, $$
 and 
 $$\ga_m : = V(x_m,\al_0) - \vphi(x_m,\al_0) \to 0, \text{ as } m \to \infty. $$
Choose $m$ sufficiently large so that $\abs{x_m-x_0} < \e/2$. Fix some $Z\in \mathcal A_{x_m,\al_0}$ and let $ X\cd = X^{x_m,\al_0}(\cdot, Z)$ be the corresponding controlled process. 
Define $B_\e(x_0):= \set{x\in \bar S: |x-x_0| < \e}$,   where
  $\e>0$ is small enough so that (i) 
   $ \varphi(x,\al)\ge V^*(x,\al) \ge V(x,\al)$ for all $(x,\al) \in \lbar B_\e(x_0)\times \M$, and (ii)
   \begin{equation}\label{vis-supersoln-contra}
 \min\set{(r-\op) \varphi(x,\al),\min_{i=1,\dots,n}\set{D_{x_i}\varphi(x,\al)-f_i(x,\al)}} \ge  A >0, \ \ \forall (x,\al) \in \lbar B_\e(x_0)\times \M.
 \end{equation}

 Let $\theta_m:=\inf\set{t\ge 0:  X(t) \notin B_\e(\xz)}$.  
 Then for any $t>0$, we have from the generalized It\^o formula that 
 \begin{displaymath}
\begin{aligned}
\varphi(x_m,\al_0)  = &\  \ex \left[e^{-r (t\wedge \theta_m)}\varphi( X( t\wedge \theta_m-),\al( t\wedge \theta_m-))+ \int_0^{ t\wedge \theta_m-} e^{-rs} (r- \op) \varphi(X(s),\al(s))ds \right]
  \\ &  
 + \ex\left[ \int_0^{ t\wedge \theta_m-} e^{-rs} \sum_{i=1}^n D_{x_i} \varphi( X(s),\al(s))dZ^c_i(s)\right]\\ &  - \ex\left[\sum_{0\le s <  t\wedge \theta_m}  e^{-rs}\left[\varphi(X(s),\al(s-))-\varphi(X(s-),\al(s-))\right]\right].
\end{aligned}
\end{displaymath}
  Note that \bed \begin{aligned}& \varphi(  X(s),\al(s-))-\varphi(  X(s-),\al(s-)) \\ & \ \ = \sum^n_{i=1} (  X_i(s)-  X_i(s-)) D_{x_i}\varphi(  X(s-)+z(  X(s)-  X(s-)),\al(s-))\\ &\ \ = - \sum^n_{i=1} \Delta Z_i(s) D_{x_i}\varphi(  X(s-)+z(  X(s)-  X(s-)),\al(s-)) \end{aligned} \eed for some $z\in[0,1]$.
  But by virtue of \eqref{vis-supersoln-contra}, for all $0\le s <t\wedge \theta_m$, we have
  \bed D_{x_i}\varphi(  X(s-)+z(  X(s)-  X(s-)),\al(s-)) \ge f_i (  X(s-)+z(  X(s)-  X(s-)),\al(s-))+ A.\eed
  Further, since $  X(s) \le   X(s-)+z(  X(s)-  X(s-)) \le   X(s-)$ and that $f_i(\cdot,\al)$ is non-increasing, we have
  \bed  f_i (  X(s-)+z(  X(s)-  X(s-)),\al(s-)) \ge f_i(  X(s-),\al(s-)).\eed
 Then using  \eqref{vis-supersoln-contra} again, we obtain
 \begin{equation}\label{eq-vphi-new}
\begin{aligned}
\varphi(x_m,\al_0)  \ge &\  \ex \left[e^{-r (t\wedge \theta_m)}\varphi( X( t\wedge \theta_m-),\al( t\wedge \theta_m-))
                    + \int_0^{ t\wedge \theta_m-} e^{-rs} A( ds + \one \cdot dZ(s))\right] \\ 
                      & \ \ + \ex\left[ \int_0^{ t\wedge \theta_m-} e^{-rs}  f(X(s-),\al(s-)) \cdot dZ(s) \right],
\end{aligned}
\end{equation} where $\one=(1,\dots, 1)'$.
Now letting $t \to \infty$  in \eqref{eq-vphi-new},  it follows that  on the set $\set{\theta_m= \infty}$, we have 
 \begin{equation}\label{Vm-estimate-eq1}
 \begin{aligned}
V(x_m,\al_0)& = \vphi(x_m,\al_0) + \gamma_m  \\ 
                    & \ge \ex\left[ \int_0^{ \infty} e^{-rs}  f(X(s-),\al(s-)) \cdot dZ(s) \right] + \frac{A}{r} + \gamma_m.
\end{aligned}
\end{equation}

{\em Step 2.} On the set $\set{\theta_m < \infty}$, we have by letting $t\to \infty$ in \eqref{eq-vphi-new} that 
 \beq\label{super-ineq1} \begin{aligned}
\varphi(x_m,\al_0) \ge 
  &\  \ex \left[e^{-r\theta_m}\varphi(  X(\theta_m-),\al(\theta_m-))+  \int_0^{\theta_m-}  e^{-rs}  f(  X(s),\al(s)) \cdot dZ(s) \right]
 \\ &   + A \ex\left[ \int_0^{\theta_m-}  e^{-rs} (ds+ 
 \one \cdot dZ(s)) \right].
   \end{aligned}\eeq 
 Note that $  X(\theta_m) \le   X(\theta_m-)$ and $  X(\theta_m-) \in B_\e(\xz)$. Thus there exists some $\lambda\in [0,1]$ such that
  $$x_\la:=   X(\theta_m-) + \la (  X(\theta_m)-  X(\theta_m-)) =   X(\theta_m-) - \la \Delta Z(\theta_m) \in \partial B_\e(\xz).$$
  Moreover, $  X(\theta_m) \le x_\la \le   X(\theta_m-)$.
  Note that $$ \begin{aligned} \varphi & (  X(\theta_m-),\al(\theta_m-)) - \varphi(x_\la,\al(\theta_m-)) \\ & =
          (X(\theta_m-)-x_\la) \cdot D\varphi(  X(\theta_m- )+ z(  X(\theta_m)-x_\la),\al(\theta_m-))  \\ & = \la \Delta  Z(\theta_m) \cdot D\varphi(  X(\theta_m- )+ z(  X(\theta_m)-x_\la),\al(\theta_m-)) .\end{aligned}$$
  But (\ref{vis-supersoln-contra}) and the monotonicity of $f_i(\cdot, \al)$ imply that
   $$\begin{aligned}  D_{x_i} \varphi(  X(\tha_m- )+ z(  X(\tha_m)-x_\la),\al(\tha_m-)) & \ge f_i(  X(\tha_m- )+ z(  X(\tha_m)-x_\la),\al(\tha_m-)) +A\\ & \ge f_i(  X(\tha_m- ),\al(\tha_m-))+A.\end{aligned}$$
   This, together with the fact that $\Delta Z_i(\tha_m)\ge 0$, leads to
   \beq\label{super-ineq2} \varphi(  X(\theta_m-),\al(\theta_m-)) - \varphi(x_\la,\al(\theta_m-)) \ge \la \Delta  Z(\tha_m) \cdot \( f(  X(\tha_m- ),\al(\tha_m-))+A\one\).\eeq
   Combing (\ref{super-ineq1}) and  (\ref{super-ineq2}), we obtain
   \beq\label{super-ineq3}\begin{aligned}
   V(x_m, \al_0) & = \varphi(x_m, \al_0) + \ga_m
   \\ & \ge  \ex  \left[ \int_0^{\theta_m-}  e^{-rs}  f(  X(s-),\al(s-)) \cdot dZ(s) 
   + A  \int_0^{\theta_m-}  e^{-rs} \(ds+
   \one \cdot dZ(s)\) \right] 
   \\ &\ \  + \ex\left[ \la e^{-r\tha_m}  \Delta Z(\tha_m) \cdot f(  X(\tha_m- ),\al(\tha_m-))  \right]  
   \\ & \ \ + \ex\left[   e^{-r\tha_m} \(\varphi(x_\la, \al(\tha_m-))+ \la A \Delta Z(\theta_m) \cdot \one\)\right]  + \ga_m.
   \end{aligned} \eeq

     Note that $x_\la \in \lbar B_\e(\xz)$ and hence  $\varphi(x_\la, \al(\theta_m-)) \ge V^*(x_\la, \al(\theta_m-))$.
   On the other hand, since $  X(\theta_m)\le x_\la \le X(\theta_m-)$, it follows  from \eqref{Vstar} and the monotonicity of $f$ that
   \beq\label{super-ineq4}\begin{aligned}
       V^*(x_\la,\al(\theta_m-))&  \ge V^*(  X(\theta_m),\al(\theta_m-))+  (x_\la-  X(\theta_m)) \cdot  f(x_\la,\al(\theta_m-)) 
    \\ & \ge V^*(  X(\theta_m),\al(\theta_m-)) + (1-\la)  \Delta Z(\theta_m) \cdot f(  X(\theta_m-),\al(\theta_m-)).
   \end{aligned} \eeq
A similar argument as that in \cite{Song-S-Z} yields that   \beq\label{super-claim1}\ex\left[ e^{-r\theta_m} V^*(  X(\theta_m), \al(\theta_m-)) \right]= \ex \left[e^{-r\theta_m} V^*( X(\theta_m), \al(\theta_m))\right].\eeq
   Now put (\ref{super-ineq4}) and (\ref{super-claim1}) into  (\ref{super-ineq3}) and  we obtain
    \beq\label{super-ineq5}\begin{aligned}
   V(x_m,\al_0) &  \ge  \ex  \left[\int_0^{\theta_m-}  e^{-rs}  f( X(s-),\al(s-)) \cdot dZ(s) 
     +  e^{-r\theta_m} V^*( X(\theta_m), \al(\theta_m))\right]
   \\ &\ \  + A \ex      \left[\int_0^{\theta_m-}  e^{-rs}\( ds
 +  \one \cdot dZ(s)\)\right] + \ga_m \\ & \ \   +  (1-\la)\ex \left[e^{-r\theta_m} \Delta Z(\theta_m)\cdot f(  X(\theta_m-),\al(\theta_m-))  \right] \\
  & \ \ +  \la \ex\left[ e^{-r \theta_m} \Delta Z(\theta_m) \cdot \(f( X(\theta_m- ),\al(\theta_m-))+A \one\)\right]
  \\ & =  \ex \left[ \int_0^{\theta_m}  e^{-rs}  f(  X(s-),\al(s-))\cdot dZ(s)  +   e^{-r\theta_m} V^*( X(\theta_m), \al(\theta_m))\right]\\
  & \ \ +   A \ex      \left[\int_0^{\theta_m-}  e^{-rs}\( ds+ 
   \one\cdot dZ(s)\) + \la  e^{-r\theta_m}   \one \cdot  \Delta Z(\theta_m) \right] + \ga_m.
    \end{aligned} \eeq

  We   now claim that for  some constant $\kappa >0 $  that does not depend on $m$, we have
  \begin{equation}\label{super-claim2}\ex      \left[\int_0^{\theta_m-}  e^{-rs} \(ds
 +  \one \cdot dZ(s)\) + \la   e^{-r\theta_m}   \one \cdot \Delta Z(\theta_m) \right] \ge \kappa. \end{equation} .

{\em Step 3.}   Assume   (\ref{super-claim2}) for the moment. Then  (\ref{super-ineq5}) can be rewritten as
   \beq\label{super-ineq6}\begin{aligned}
   V(x_m,\al_0) &  \ge  \ex \left[ \int_0^{\theta_m}  e^{-rs}  f(  X(s-),\al(s-))\cdot dZ(s) 
     +   e^{-r\theta_m} V^*(  X(\theta_m), \al(\theta_m))\right] + A \kappa + \gamma_m.  \end{aligned} \eeq
  Combining \eqref{Vm-estimate-eq1} and   \eqref{super-ineq6}, and then taking supremum 
  over $Z\in \mathcal A_{x_m,\al_0}$,
      it follows  from the weak dynamic programming principle  \eqref{DPP-eq1} that
     \bed  V(x_m,\al_0) \ge  V(x_m,\al_0) + \frac{A}{r}\wedge A\kappa + \ga_m >  V(x_m,\al_0),\eed
    for $m$ sufficiently large.  This is a contradiction. So we must have (\ref{viscosity-super}) and hence $V$ is a viscosity subsolution of (\ref{qvi-equiv-form}).

{\em Step 4.}      Now it remains to show  (\ref{super-claim2}). To this end, we consider the function $\wdt W(x,\al):= \abs{x-x_0}^2 -\e^2$ for $(x,\al)\in B_\e(\xz)\times \M$. Then it follows that
     \bed (\op-r) \wdt W(x,\al)= 2 ( x-\xz) \cdot b(x,\al) + {1\over 2}\tr(2I \sigma(x,\al)\sigma'(x,\al)) -r ( \abs{x-x_0}^2 -\e^2).\eed
 Since $\wdt W$, $b$, and $\sigma$ are continuous, and $\M$ is a finite set, it is obvious  that
 $$| (\op-r) \wdt W(x,\al)| \le K < \infty$$ for some positive constant $K$. Now let $K_0:=\frac{1}{2\e+ K}$ and define
 $W(x,\al)= K_0\wdt W(x,\al)$ for  $(x,\al)\in B_\e(\xz)\times \M$.
 Then it follows immediately that \beq\label{claim2-ineq1}\abs{ (\op-r)   W(x,\al)} < 1, \ \ (x,\al)\in B_\e(\xz)\times \M. \eeq
 Moreover, we have
 \beq\label{claim2-ineq2} D_{x_i} W(x,\al) = 2K_0 ( x-\xz) \cdot e_i  \ge -1,\eeq
 where $e_i= (0,\dots, 1,\dots,0)'$ denotes the $i$th unit vector. Let $x_m$, $\theta_m$, $Z\in \mathcal A_{x_m,\al_0}$ etc. as before. 
Using \eqref{claim2-ineq1},  \eqref{claim2-ineq2}, and generalized It\^o's formula, detailed computations similar to those in Step 1 yield 
\beq\label{claim2-ineq4}  \ex  \left[e^{-r\theta_m} W(  X(\theta_m-),\al(\theta_m-))\right] - W(x_m,\al_0) 
   \le  \ex \left[\int_0^{\theta_m-} e^{-rs} \(ds +   \one 
 \cdot  dZ(s)\)\right].  \eeq

Also, recall that $  X(\theta_m) \le x_\la \le   X(\theta_m-)$. It follows from (\ref{claim2-ineq2}) that
\beq\label{claim2-ineq5} \begin{aligned}
 W&(X(\theta_m-),\al(\theta_m-)) - W(x_\la,\al(\theta_m-))\\ & = \sum^n_{i=1} D_{x_i} (x_\la + z(\hat X(\theta_m-)-x_\la),\al(\theta_m-)) 
    (X(\theta_m-)-x_\la) \cdot e_i 
\\ &=\la \sum^n_{i=1} D_{x_i} (x_\la + z(\hat X(\theta_m-)-x_\la),\al(\theta_m-))\Delta Z_i(\theta_m) \\ &  \ge -\la \sum^n_{i=1} \Delta  Z_i(\theta_m) = -\la  \one \cdot  \Delta Z(\theta_m) .
\end{aligned}\eeq
Combining (\ref{claim2-ineq4}) and (\ref{claim2-ineq5}), we have
\bed  \ex \left[\int_0^{\theta_m-} \!\!e^{-rs} \(ds +   \one 
\cdot  dZ(s)\)+ \la e^{-r\theta_m}\one\cdot \Delta Z(\theta_m)\right] \ge \ex \left[e^{-r\theta_m}W(x_\la,\al(\theta_m-)) \right]-W(x_m,\al_0). \eed
But $x_\la \in \partial B_\e(\xz)$, and consequently $W(x_\la,\al(\theta_m-))=0$. Also, it is immediate that $W(x_m,\al_0)=K_0(\abs{x_m-\xz}^2-\e^2) \le K_0((\frac{\e}{2})^2-\e^2) = -\frac{3}{4}K_0 \e^2$. Hence it follows that
\bed \ex \left[\int_0^{\theta_m-} \!\!e^{-rs} \(ds +   \one 
\cdot  dZ(s)\)+ \la e^{-r\theta_m}\one\cdot \Delta Z(\theta_m)\right] \ge  \frac{3}{4}  K_0\e^2 =:\kappa >0.\eed
This establishes (\ref{super-claim2}) and hence finishes the proof of the theorem.
  \end{proof}

\section{Viscosity Solution: Uniqueness}\label{sect-strong-comparison}

Our goal is to establish a strong comparison result for constrained viscosity solutions of \eqref{qvi-equiv-form}.
To this end, we need the following lemma.

\begin{lem}\label{lem-equivalent-comparison} Let $s(x) = x \cdot \one=\sum_{i=1}^n x_i $ and \beq\label{UV-defn}
\tilde u(x,\al):= e^{-\lambda s(x)} u(x,\al),\ \  \tilde v(x,\al):= e^{-\lambda s(x)} v(x,\al),\ \forall (x,\al) \in S\times \M,
\eeq where $\la>0$. 
Then
\begin{enumerate}
 \item[{\em (a)}] $u(x,\al)$ is viscosity subsolution of \eqref{qvi-equiv-form} if and only if $\tu(x,\al)$ is a viscosity subsolution of
\beq\label{subsoln-tu}\begin{aligned}
\min\big\{r\tu(x,\al)-H_\la(x,\al, \tu(x,\al), D\tu(x,\al), D^2 \tu(x,\al))- Q\tu(x,\cdot)(\al)   , \\  \min_{i=1,\dots,n}\set{e^{\lambda s(x)}[\la \tu(x,\al)+ D_{x_i}\tu(x,\al)]-f_i(x,\al)}\big\} = 0, \end{aligned}
\eeq
where for any $(x,\al, q, p, A)\in \rr^n \times \M   \times \rr \times \rr^n \times \mathcal S_n$,
 \bed \begin{aligned}H_\la(x,\al, q, p, A) = &\frac{1}{2}\tr(\sg\sg'(x,\al) A) +\frac{ \la}{2}  \(\one'\sg\sg'(x,\al) p + p'\sg\sg'(x,\al) \one \)
 \\ & \quad +  b(x,\al) \cdot p +  \la q b(x,\al) \cdot \one +      \frac{\la^2}{2} q \abs{\sg'(x,\al)\one}^2,\end{aligned} \eed
and
\bed Q \tu(x, \cdot)(\al)= \sum_{j=1}^m q_{\al j} \tu(x,j)= \sum_{j=1}^m q_{\al j}[ \tu(x,j)-\tu(x,\al)].\eed

\item[{\em (b)}] Similarly, $v(x,\al)$ is viscosity supersolution of \eqref{qvi-equiv-form} if and only if $\tv(x,\al)$ is a viscosity supersolution of
\eqref{subsoln-tu}.
\end{enumerate}
\end{lem}

\begin{proof} We prove part (a) only; the proof of part (b) is similar. Suppose $u$ is viscosity subsolution
 of \eqref{qvi-equiv-form}. 
   Let $\tilde \vphi (\cdot, \al)\in C^2, \al \in \M$ and
 let $(\xalz)$ be a maximum point of $\tu-\tilde \vphi$ with $(\tu-\tilde \vphi)(\xalz)=0$.
 Put $$\vphi(x,\al):= e^{\la s(x)} \tvphi(x,\al).$$ Then
 it is easy to verify that
 $\vphi(\cdot, \al)\in C^2, \al\in \M$  and $(\xalz)$ is a maximum point of
 $u-\vphi$ with $(u-\vphi)(\xalz)=0$.
 Since $u$ is viscosity subsolution of \eqref{qvi-equiv-form}, we obtain
 \beq\label{u-sub-eq1}
r\vphi(\xalz) - \frac{1}{2} \tr(\sg\sg'(\xalz)D^2\vphi(\xalz)) - b(\xalz)\cdot D\vphi(\xalz)
 - \sum_{j=1}^m q_{\al_0 j}\vphi(x_0,j)   \le 0,\eeq
or \beq\label{u-sub-eq2}
\min_{i=1,\dots, n}\set{D_{x_i} \vphi(\xalz)-f_i(\xalz)}\le 0.\eeq

Since $\vphi(x,\al) = e^{\la s(x)} \tvphi(x,\al)$, we compute
\bed
D_{x_i} \vphi(x,\al)= e^{\la s(x)}[\la  \tvphi(x,\al)+ D_{x_i} \tvphi(x,\al)],
\eed
and
\bed \begin{aligned}\Dij \vphi(x,\al) = & e^{\la s(x)}
\big[    \la^2   \tvphi(x,\al)
   +  \la (  \Dj \tvphi(x,\al)  +     D_{x_i} \tvphi(x,\al)) + \Dij \tvphi(x,\al)\big].
\end{aligned}\eed
In other words,
\bed D\vphi(x,\al)= e^{\la s(x)}\left[ \la  \tvphi(x,\al) \one  +  D  \tvphi(x,\al)\right],\eed
and
\bed \begin{aligned}
D^2 \vphi(x,\al)=e^{\la s(x)} & \big[   \la^2 \tvphi(x,\al) \one \one'  + \la ( \one  D \tvphi(x,\al)'  +     D  \tvphi(x,\al) \one')+ D^2 \tvphi(x,\al) \big].
\end{aligned}\eed
Then substituting $D\vphi$ and $D^2 \vphi$ into \eqref{u-sub-eq1} leads to
\bed\begin{aligned}
0& \ge  r\tvphi(\xalz)  - \sum_{j=1}^m q_{\al_0 j} \tvphi(x_0,j) -     b(\xalz)\cdot(\la  \tvphi(x_0,\al_0) \one +  D  \tvphi(x_0,\al_0))  \\ & \ \ -
\frac{1}{2}\tr\(\sg\sg'(\xalz)\left[ \la^2 \tvphi(x_0,\al_0) \one \one'   + \la (\one  D \tvphi(x_0,\al_0)'  +     D  \tvphi(x_0,\al_0) \one')+ D^2 \tvphi(x_0,\al_0) \right]\)  ,
\end{aligned}\eed
which can be   rewritten as
\beq\label{u-sub-eq11} \begin{aligned}
 r  \tvphi(\xalz) - H_\la(\xz, \al_0, \tvphi(\xalz), D  \tvphi(\xalz), D^2  \tvphi(\xalz)) - Q\tvphi(\xz, \cdot) (\al_0)   \le 0.\end{aligned}\eeq
Similarly,
\eqref{u-sub-eq2} can be rewritten as
\beq\label{u-sub-eq21}
\min_{i=1,\dots,n} \set{e^{\la s(\xz)}[\la  \tvphi(\xalz)+ D_{x_i} \tvphi(\xalz)]-f_i(\xalz)} \le 0.
\eeq
Therefore in view of \eqref{u-sub-eq11} and \eqref{u-sub-eq21}, $\tu$ is a viscosity subsolution of \eqref{subsoln-tu}.

Conversely, let $\tu $ be a viscosity subsolution of \eqref{subsoln-tu}.
Recall $u(x,\al)= e^{\la s(x)}\tu(x,\al)$.
 Let $\vphi(\cdot,\al)\in C^2,\al\in\M $ and $(\xalz)$ be a maximum point of $u-\vphi$ with $(u-\vphi)(\xalz)=0$.
 Put $\tvphi(x,\al):=e^{-\la s(x)}\vphi(x,\al)$.
 Detailed calculations as above
  show  that $u$ is a viscosity subsolution of \eqref{qvi-equiv-form}.
\end{proof}

\begin{lem}\label{lem-Atar}
For every $\xi\in \bar S$, there exist $\eta=\eta(\xi)\in \rr^n$ and $a=a(\xi) >0$ such that 
\begin{displaymath}
B_{ta}(x+ t \eta) \subset S, \ \ \forall x\in \bar S\cap B_a(\xi), \ \forall t\in (0,1],
\end{displaymath} where $B_a(x)=\set{y\in \rr^n: |y-x| < a}$. 
\end{lem}
\begin{proof}
See \cite{MR2271486}.
\end{proof}
With Lemmas \ref{lem-equivalent-comparison}  and \ref{lem-Atar} at our hands, we are now ready to establish the strong comparison result for the constrained viscosity solution of \eqref{qvi-equiv-form}.

 \begin{thm}\label{thm-comparison} 
 Let $\mathbf u\in USC(\bar S\times \M; \rr^m) $ and $\mathbf v\in LSC(\bar S\times \M; \rr^m) $ be respectively   viscosity subsolution on $\bar S\times \M$ and supersolution in $S\times \M$ of
 \eqref{qvi-equiv-form} and satisfy 
\beq\label{uv-poly-growth}
\abs{u(x,\al)} + \abs{v(x,\al)} \le K (1+\abs{x}^p), \ \ \forall (x,\al) \in \bar S\times \M,
\eeq
where $K$ and $p$ are positive constants.  Assume  that  for some positive constant $\kappa_0$,  we have
\begin{equation}\label{ito-condition}  \abs{b(x,\al)-b(y,\al)} + \abs{\sigma(x,\al)-\sigma(y,\al)} \le \kappa_0 \abs{x-y},  \end{equation}
\beq\label{new-condition} 
 b(x,\al)'\one \le \kappa_0 \text{ and } \abs{\sigma(x,\al)'\one} \le \kappa_0, 
\eeq  for all  $ (x,\al)\in \bar S\times \M$.  
Then  we have $$u(x,\al) \le v(x,\al), \ \ \forall (x,\al) \in \bar  S \times \M.$$
 \end{thm}

\begin{proof} Suppose on the contrary that \beq\label{eq-M-defn} M:= \max_{\al\in \M} \sup_{x\in \bar S} [u(x,\al)-v(x,\al)] >0.\eeq
We will derive a contradiction in the following. Define $\tu$ and $\tv$ as in \eqref{UV-defn}, where $\la>0$ is a constant to be determined later.
 Thanks to \eqref{uv-poly-growth}, $\tu$ and $\tv$ are uniformly bounded.
Moreover,
we have
\bed \lim_{|x|\to \infty, \, x\in \bar S} (\abs{\tu(x,\al)} + \abs{\tv(x,\al)})=0, \  \forall \al \in \M.\eed
 Therefore in view of \eqref{eq-M-defn} and the facts that $\M$ is finite and that
 $\tu-\tv$ is upper semicontinuous, there exist  some
   bounded set $O$ of $\bar S$ and $(\hat x, \ell )\in O \times \M$, such that
 \beq\label{eq-tildeM-defn} \tilde M:= \max_{\al\in \M}\sup_{x\in \bar S}[ \tu(x,\al)- \tv(x,\al) ]
  =  \max_{x\in O}[ \tu(x,\ell)- \tv(x,\ell) ]= \tu(\hat x,\ell)- \tv(\hat x, \ell)>0. \eeq

Let $\eta=\eta(\hat x) $ be as in Lemma  \ref{lem-Atar}. For any $\e, \delta \in (0,1)$ and   
 $(x,y) \in O \times O$, define
\beq\label{phi-e-defn}
\begin{aligned}
&\Phi(x,y)=\Phi_{\e,\delta, \lambda}(x,y) := \tu(x,\ell)- \tv(y,\ell) -\phi(x,y),\\
& \phi(x,y):=\abs{\frac{1}{\e}(y-x)-\delta \eta}^2 + \delta\abs{x-\hat x}^2.
\end{aligned}\eeq
 Note that $\Phi$ is USC and hence   achieves its maximum $M=M_{\e,\delta,\lambda}$
on the compact set $\bar O^2$
at $(\tilde x, \tilde y):=(x_{\e,\delta,\lambda},y_{\e,\delta,\lambda})$.
By virtue of Lemma \ref{lem-Atar}, $\hat x + \e \delta \eta \in S^o$.
Also, since \begin{displaymath}
\begin{aligned}
\Phi(\tilde x, \tilde y)& =\tu (\tilde x, \ell)-\tv(\tilde y,\ell)- \abs{\frac{1}{\e}(\tilde y-\tilde x)-\delta\eta}^2-\delta\abs{\tilde x-\hat x}^2\\
 & \ge \Phi (\hat x, \hat x+ \e\delta \eta)
  = \tu (\hat x, \ell) - \tv(\hat x+ \e\delta \eta,\ell),
\end{aligned}
\end{displaymath} we have 
\begin{displaymath}
\tu (\tilde x, \ell)-\tv(\tilde y,\ell)-  \tu (\hat x, \ell) + \tv(\hat x+ \e\delta \eta,\ell) \ge  \abs{\frac{1}{\e}(\tilde y-\tilde x)-\delta\eta}^2+\delta\abs{\tilde x-\hat x}^2.
\end{displaymath}
 Multiplying $\e^2$ on both sides of the above equation, we see that for each $\delta $ and $\lambda$, $\tilde x - \tilde y \to 0$ as $\e \to 0$.  Further, by virtue of   \eqref{eq-tildeM-defn}, we have 
 \begin{displaymath}
\limsup_{\e \to 0}  \abs{\frac{1}{\e}(\tilde y-\tilde x)-\delta\eta}^2+\delta\abs{\tilde x-\hat x}^2 \le 0;
\end{displaymath}
and therefore 
\begin{equation}\label{eq-tx-to-hx}
\tilde x \to \hat x, \text{ and } \frac{1}{\e}(\tilde y-\tilde x)\to\delta\eta, \ \text{ as }\e\to 0.
\end{equation}
In particular, it follows that 
 \begin{equation}
\label{eq-ty-to-hx}
\tilde y = \tilde x + \e\delta \eta + o(\e) = \hat x + \e\delta \eta + o(\e), \text{ as }\e\to 0,
\end{equation}
and hence $\tilde y \in S^o$ for $\e$ sufficiently small. 

     
The function $x\mapsto \tu(x,\ell) - \phi_1(x)$ achieves its maximum at $\tilde x$, where
 $$\phi_1(x)= \tv(\tilde y,\ell) + \abs{ \frac{1}{\e}(\tilde y-x) -\delta\eta}^2 + \delta \abs{x-\hat x}^2.$$
Moreover, we compute
$$D \phi_1(x)= -\frac{2}{\e}\(\frac{1}{\e}(\tilde y- x)-\delta\eta\)+2\delta(x-\hat x), \text{ and } D^2 \phi_1(x)= \frac{2}{\e^2}I+2\delta I. $$
Hence it follows   from \lemref{lem-equivalent-comparison}, the definition of viscosity subsolution, and Ishii's lemma  that for some $M\in \SS_n$, $(-D\phi_1(\tilde x), M) \in \bar {\mathcal P}^{2,+}\tu(\tilde x,\ell)$,  such that
\bea\ad \min\Big\{r \tu(\tilde x,\ell)-H_\lambda(\tilde x, \ell, \tu(\tilde x,\ell), D\phi_1(\tilde x),M) - Q\tu(\tilde x ,\cdot)(\ell),  
 \\ \aad \qquad 
 \ \ \min_{i=1,\dots, n}\set{e^{\la s(\tilde x )}[\la \tu(\tilde x ,\ell)+D \phi_1(\tilde x)\cdot e_i]-f_i(\tilde x ,\ell)}\Big\} \le 0.\eea
Thus either  
\beq\label{comparison-case1}
\min_{i=1,\dots, n}\set{e^{\la s(\tilde x )}[\la \tu(\tilde x ,\ell)+D \phi_1(\tilde x)\cdot e_i]-f_i(\tilde x ,\ell)} \le 0,
\eeq
or
\beq\label{comparison-case2}
r \tu(\tilde x ,\ell)-H_\lambda(\tilde x , \ell, \tu(\tilde x ,\ell),D\phi_1(\tilde x),M) - Q\tu(\tilde x ,\cdot)(\ell) \le 0.
\eeq
On the other hand, the function $y\mapsto \tv(y,\ell)- \phi_2(y)$ achieves its minimum
at $\tilde y$, where $$ \phi_2(y)= \tu(\tilde x, \ell)- \(\abs{\frac{1}{\e}(y-\tilde x)-\delta \eta}^2+\delta\abs{\tilde x-\hat x}^2\).$$
Direct calculations reveal that
$$D \phi_2(y)= -\frac{2}{\e}\(\frac{1}{\e}(y-\tilde x)-\delta \eta\), \text{ and } D^2 \phi_2(y_\e)= -\frac{2}{\e^2} I. $$
Hence the definition of supersolution and Ishii's lemma imply that for some $N \in \SS_n$, we have
$(D\phi_2(\tilde y), N) \in \bar {\mathcal P}^{2,-}\tv(\tilde y,\ell)$ and
\begin{equation}\label{eq-case2}\barray\ad \min\Big\{r \tv(\tilde y,\ell)-H_\lambda(\tilde y, \ell, \tv(\tilde y,\ell),D\phi_2(\tilde y),N) - Q\tv(\tilde y,\cdot)(\ell),  \\ \aad \qquad
 \min_{i=1,\dots, n}\set{e^{\la s(\tilde y)}[\la \tv(\tilde y,\ell)+ D\phi_2(\tilde y) \cdot e_i]-f_i(\tilde y,\ell)}\Big\} \ge 0.
 \earray
 \end{equation}

{\em Case 1.} Now suppose \eqref{comparison-case1} is true.
Recall $\tu(x,\al)= e^{-\la s(x)} u(x,\al)$ and $\tv(x,\al)= e^{-\la s(x)} v(x,\al)$.
Then we have from \eqref{comparison-case1} and \eqref{eq-case2} that 
\bea
0 \ad \ge \min_{i=1,\dots, n} \left\{e^{\la s(\tilde x)}[\la \tu(\tilde x,\ell)+D\phi_1(\tilde x)\cdot e_i]-f_i(\tilde x,\ell) 
  - e^{\la s(\tilde y)}[\la \tv(\tilde y,\ell)+D\phi_2(\tilde y)\cdot e_i]+f_i(\tilde y,\ell)\right\}\\
\ad = \min_{i=1,\dots, n} \set{\la(u(\tilde x,\ell)-v(\tilde x,\ell))+ (e^{\la s(\tilde x)}D\phi_1(\tilde x)-e^{\la s(\tilde y)} D\phi_2(\tilde y)) \cdot e_i- (f_i(\tilde x,\ell)-f_i(\tilde y,\ell)) }.
\eea
Hence it follows that
\begin{equation}\label{comparison-case1-eq1}
\begin{aligned}
\lefteqn{\la(u(\tilde x,\ell)-v(\tilde x,\ell)) } \\ & \le  \max_{i=1,\dots, n} \set{ (f_i(\tilde x,\ell)-f_i(\tilde y,\ell)) -(e^{\la s(\tilde x)}D\phi_1(\tilde x)-e^{\la s(\tilde y)} D\phi_2(\tilde y)) \cdot e_i }\\ 
 & =  \max_{i=1,\dots, n} \Bigg\{ (f_i(\tilde x,\ell)-f_i(\tilde y,\ell))   -e_i \cdot \Bigg[e^{\la s(\tilde x)} \left(-\frac{2}{\e}\(\frac{1}{\e}(\tilde y- \tilde x)-\delta\eta\)+2\delta(\tilde x-\hat x) \right)  
 \\ & \hfill \qquad \qquad \hspace{2.5in}
 + e^{\la s(\tilde y)}   \frac{2}{\e}\(\frac{1}{\e}(\tilde y-\tilde x)-\delta \eta\)\Bigg]\Bigg\}.
\end{aligned} \end{equation} 

Thanks to \eqref{eq-tx-to-hx},  \eqref{eq-ty-to-hx}, and the continuity of $f_i$, the
right-hand-side of \eqref{comparison-case1-eq1} converges to 0 as $\e
\downarrow 0$.
Note that for any $x\in O$, we have
\bea \tu(x,\ell)-\tv(x,\ell)\ad  = \Phi(x,x)  \le \Phi(\tilde x,\tilde y) \\ \ad = \tu(\tilde x,\ell)-\tv(\tilde y,\ell)- \phi(\tilde x,\tilde y)
\le \tu(\tilde x,\ell)-\tv(\tilde y,\ell).  \eea
In particular, taking $x= \hat x$ 
 leads to
  $\tilde u(\hat x, \ell)- \tilde v(\hat x, \ell)\le 0$, which gives a contradiction
  to \eqref{eq-tildeM-defn}. Hence, Case 1 is impossible.

{\em Case 2.} Now suppose \eqref{comparison-case2} is true. Then
it follows from \eqref{comparison-case2} and \eqref{eq-case2} that
\beq\label{comparison-case2-eq1}\begin{aligned}
\ad r(\tu(\tilde x,\ell) -\tv(\tilde y,\ell))  - \left[Q\tu(\tilde x,\cdot)(\ell)-Q\tv(\tilde y,\cdot)(\ell) \right]
\\ \aad \hfill  \ \ - \left [H_\lambda(\tilde x, \ell, \tu(\tilde x,\ell),D\phi_1(\tilde x),M)-H_\lambda(\tilde y, \ell, \tv(y_\e,\ell),D\phi_2(\tilde y),N)  \right] \le 0.
\end{aligned}\eeq
Using the definition of $H_\la$, 
\bea
\ad H_\lambda(\tilde x, \ell, \tu(\tilde x,\ell),D\phi_s(\tilde x),M)-H_\lambda(\tilde y, \ell, \tv(\tilde y,\ell),D\phi_2(\tilde y),N) \\[1.5ex]
\aad = \frac{1}{2}\( \tr(\sigma\sigma'(\tilde x,\ell)M)-\tr(\sigma\sigma'(\tilde y,\ell)N) \)  + \la\( \tu(\tilde x,\ell)b(\tilde x,\ell)   -  \tv(\tilde y,\ell)b(\tilde y,\ell)\)\cdot \one
 \\[1.5ex] \aad \  + \frac{\la}{2} \(\left[\one'\sigma\sigma'(\tilde x,\ell)D\phi_1(\tilde x) + D\phi_1(\tilde x)'\sigma\sigma'(\tilde x,\ell)\one \right]  - \left[\one'\sigma\sigma'(\tilde y,\ell)D\phi_2(\tilde y) + D\phi_2(\tilde y)'\sigma\sigma'(\tilde y,\ell)\one\right]  \)
\\[1.5ex] \aad \   +
b(\tilde x,\ell)\cdot D\phi_1(\tilde x) -  b(\tilde y,\ell)\cdot D\phi_2(\tilde y)
+ \frac{\la^2}{2} \( \tu(\tilde x,\ell)  \abs{\sg'(\tilde x,\ell)\one}^2
  -    \tv(\tilde y,\ell)  \abs{\sg'(\tilde y,\ell)\one}^2\).
\eea
Hence it follows that 
\beq\label{eq5-20}\begin{aligned}
\!\! \lefteqn{\left[r  - \lambda b(\tilde x,\ell)'\one -\frac{\lambda^2}{2} \abs{\sigma'(\tilde x,\ell)\one}^2\right](\tu(\tilde x,\ell)-\tv(\tilde y,\ell)) } & \\ 
 & \ \le  \left[Q\tu(\tilde x,\cdot)(\ell)-Q\tv(\tilde y,\cdot)(\ell)\right] + \lambda \tv(\tilde y,\ell)(b(\tilde x,\ell)-b(\tilde y,\ell))\cdot \one 
 \\& \ \ + \frac{1}{2}\left[ \tr(\sigma\sigma'(\tilde x,\ell)M)-\tr(\sigma\sigma'(\tilde y,\ell)N) \right]  \\ 
 &\ \  +\frac{\la}{2} \(\left[\one'\sigma\sigma'(\tilde x,\ell)D\phi_1(\tilde x) + D\phi_1(\tilde x)'\sigma\sigma'(\tilde x,\ell)\one \right]  - \left[\one'\sigma\sigma'(\tilde y,\ell)D\phi_2(\tilde y) + D\phi_2(\tilde y)'\sigma\sigma'(\tilde y,\ell)\one\right]  \)
\\  &\ \  +
b(\tilde x,\ell)\cdot D\phi_1(\tilde x) -  b(\tilde y,\ell)\cdot D\phi_2(\tilde y)
+ \frac{\la^2}{2} \tv(\tilde y,\ell) \(  \abs{\sg'(\tilde x,\ell)\one}^2
  -     \abs{\sg'(\tilde y,\ell)\one}^2\).
\end{aligned}\eeq 

Using \eqref{new-condition} and the fact that $r>0$, we choose $\la >0 $ sufficiently small so that 
\beq\label{u-v-pos-coeff}r  - \lambda b(\tilde x,\ell)'\one -\frac{\lambda^2}{2} \abs{\sigma'(\tilde x,\ell)\one}^2 > 0.\eeq 
Next we analyze the terms on the right-hand side of \eqref{eq5-20}. 
 Recall that 
\begin{displaymath}
\Phi(\tilde x,\tilde y)= \tu(\tilde x, \ell)-\tv(\tilde y,\ell)-\phi(\tilde x,\tilde y) \ge \Phi(\hat x,\hat x)= \tu(\hat x, \ell)-\tv(\hat x,\ell) - \delta^2 \abs{\eta}^2.
\end{displaymath} 
Note also $q_{\ell\ell} < 0$. Thus it follows that 
\begin{displaymath}
q_{\ell\ell}\left[\tu(\tilde x,\ell)-\tv(\tilde y,\ell)\right] \le q_{\ell\ell} \left[ \tu(\hat x, \ell)-\tv(\hat x,\ell) + \phi(\tilde x,\tilde y) - \delta^2 \abs{\eta}^2 \right].
\end{displaymath}This, together with \eqref{eq-tildeM-defn},  \eqref{eq-tx-to-hx},    \eqref{eq-ty-to-hx},  
and the fact that $\tu-\tv$  is USC, lead to 
\beq\label{Q-Q-to-0} \begin{aligned} \ad\lefteqn{ \limsup_{\e\downarrow 0}\left[Q\tu(\tilde x,\cdot)(\ell)-Q\tv(\tilde y,\cdot)(\ell)\right]} \\
 \aad = \limsup_{\e\downarrow 0} \sum_{j\not =\ell} q_{\ell j}(\tu(\tilde x,j)-\tv(\tilde y,j))+ q_{\ell\ell}(\tu(\tilde x,\ell)-\tv(\tilde y,\ell))\\
 \aad \le \sum_{j\not =\ell} q_{\ell j}(\tu(\hat x,j)-\tv(\hat x,j))+ q_{\ell\ell}(\tu(\hat x,\ell)-\tv(\hat x,\ell)) -\delta^2 q_{\ell\ell}\abs{\eta}^2 \\
\aad \le \sum_{j\not =\ell} q_{\ell j} (\tu(\hat x,\ell)-\tv(\hat x,\ell))  +  q_{\ell\ell}(\tu(\hat x,\ell)-\tv(\hat x,\ell))  -\delta^2 q_{\ell\ell}\abs{\eta}^2  =  - \delta^2 q_{\ell\ell}\abs{\eta}^2. \end{aligned}\eeq
By virtue of Ishii's lemma, 
\beq\label{eq-ishii}
\lim_{\e\downarrow 0}  \frac{1}{2}\( \tr(\sigma\sigma'(\xe,\ell)M)-\tr(\sigma\sigma'(\ye,\ell)N) \) =0. 
\eeq
Next, using \eqref{ito-condition} and \eqref{eq-tx-to-hx}, and noting that $\tv$ is bounded, we have 
\beq\label{bs-bs-to-0}   \lim_{\e\downarrow 0} \left[ \lambda \tv(\tilde y,\ell)(b(\tilde x,\ell)-b(\tilde y,\ell))\cdot \one + \frac{\la^2}{2} \tv(\tilde y,\ell) \(  \abs{\sg'(\tilde x,\ell)\one}^2
  -     \abs{\sg'(\tilde y,\ell)\one}^2\) \right]=0.\eeq 
  Similarly \eqref{ito-condition} and \eqref{eq-tx-to-hx} imply that 
  \begin{equation}\label{eq-sec5-25}
\begin{aligned}
 &  \left[\one'\sigma\sigma'(\tilde x,\ell)D\phi_1(\tilde x) + D\phi_1(\tilde x)'\sigma\sigma'(\tilde x,\ell)\one \right]  - \left[\one'\sigma\sigma'(\tilde y,\ell)D\phi_2(\tilde y) + D\phi_2(\tilde y)'\sigma\sigma'(\tilde y,\ell)\one\right]   \\ 
 & \ = -\frac{2}{\e}\one' \(\sigma\sigma'(\tilde x,\ell) -\sigma\sigma'(\tilde y,\ell)\)\cdot \(\frac{1}{\e}(\tilde y-\tilde x)-\delta \eta\) \\ & \quad  -\frac{2}{\e}\(\frac{1}{\e}(\tilde y-\tilde x)-\delta \eta\)' \(\sigma\sigma'(\tilde x,\ell) -\sigma\sigma'(\tilde y,\ell)\)\one \\ 
 & \quad + 2\delta \(\one' \sigma\sigma'(\tilde x,\ell)  (\tilde x -\hat x) + (\tilde x -\hat x)' \sigma\sigma'(\tilde x,\ell)  \one \)\\ 
 & \ \to 0, \ \text{ as } \e \to 0;
\end{aligned}
\end{equation}
and \begin{equation}
\label{eq-sec5-26}
\begin{aligned}
\lefteqn{b(\tilde x,\ell)\cdot D\phi_1(\tilde x) - b(\tilde y,\ell) D\phi_2(\tilde y)} \\ 
 & = -\frac{2}{\e} (b(\tilde x,\ell)-b(\tilde y,\ell)) \cdot\(\frac{1}{\e}(\tilde y-\tilde x)-\delta \eta\) + 2 \delta b(\tilde x,\ell)\cdot (\tilde x-\hat x) \\ 
 & \to 0, \ \text{ as }\e\to 0.
\end{aligned}
\end{equation}  
  
 Now letting $\e \downarrow 0$ and using \eqref{u-v-pos-coeff}--\eqref{eq-sec5-26} in \eqref{eq5-20}, we conclude that for sufficiently small $\lambda$,  $$ \limsup_{\e\to 0} \tu(\tilde x,\ell)-\tv(\tilde y,\ell)  \le  -\delta^2 q_{\ell\ell}\abs{\eta}^2.$$
 But $\delta>0$ can be arbitrarily small, 
  as  argued  in Case 1, it follows that 
 $$\tu(\hat x, \ell)-\tv(\hat x,\ell) \le \tu (\xe,\ell)-\tv(\ye,\ell) \to 0, \text{ as } \e \to 0 \text{ and } \delta \to 0, $$
 which again contradicts  \eqref{eq-tildeM-defn}. Therefore 
 for any $x\in S$ and $\al \in \M$,
 we have
$u(x,\al) \le v(x,\al)$, as desired. 
\end{proof}

\begin{rem}
Note that under condition \eqref{new-condition},  the value function is  bounded above by an affine function. In fact,  for any $Z\in \mathcal A_{x,\al}$ with $(x,\al)\in S\times \M$, we have 
$$d (e^{-rt} X(t)) = e^{-rt} \left[(b(X(t),\al(t)) - rX(t)) dt + \sigma(X(t),\al(t)) dW(t) - d Z(t)\right].$$
Thus $$e^{-rt} dZ(t) = e^{-rt}(b(X(t),\al(t)) - rX(t))dt + e^{-rt}\sigma(X(t),\al(t)) dW(t)  - d(e^{-rt} X(t)),$$
from which it follows that 
\begin{displaymath}
\int_0^\infty e^{-rt} \one \cdot dZ(t) \le  \int_0^\infty e^{-rt} \one \cdot \left[ (b(X(t),\al(t)) - rX(t)) dt + \sigma(X(t),\al(t)) dW(t) \right] + \one \cdot x.
\end{displaymath}
Taking  expectations on both sides and using \eqref{new-condition}, we have
$$\mathbb{E} \int_0^\infty e^{-rt} \one\cdot d Z(t)  \le 
\mathbb{E} \int_0^\infty e^{-rt} (\kappa_0  dt + \one\cdot \sigma(X(t),\al(t))  dW(t)) + \one \cdot x
= \frac{\kappa_0}{r} + \one \cdot x.
$$
In the above, $\mathbb{E} \int_0^\infty e^{-rt} \one \cdot  \sigma(X(t),\al(t))  dW(t) = 0$
since $\one\cdot \sigma$ is uniformly bounded. Hence, it follows that 
$$V(x,\alpha) = \sup_{Z\in \mathcal A_{x,\al}} \ex\int_0^\infty e^{-rt} f(X(t-),\al(t-)\cdot dZ(t) \le \|f\|_\infty   \(\frac{ \kappa_0}{r} + \one\cdot x\).$$ 
\end{rem}

Finally we summarize the main result of this paper from Theorems \ref{thm-viscosity} and \ref{thm-comparison}. 
\begin{thm}\label{thm-main-result} 
	Assume  \eqref{ito-condition}, \eqref{new-condition}, and that $\mathcal A_{x,\al}\not=\emptyset$ for every $(x,\al) \in S\times \M$. Then the value function $V$ defined in \eqref{value} is the unique constrained viscosity solution of the system of coupled quasi-variational inequalities \eqref{qvi-equiv-form} on $\bar S \times \M$.
\end{thm}

\begin{rem}
At first look, condition \eqref{new-condition} seems rather restrictive. Simple models such as regime-switching geometric Brownian motion considered in Example \ref{exm3} are excluded. However, the following example indicates that in general, one can not remove \eqref{new-condition}; otherwise, uniqueness may not hold.  
\end{rem}

\begin{exm}\label{exm4}
Let's consider  a 1-dimensional   squared   Bessel process subject to control 
\begin{equation}
\label{BEDQ}
dX(t) = dt + 2 \sqrt{|X(t)|} dW(t)-dZ(t),
\end{equation} with reward functional $J(x,Z)= \ex_x\int_0^\infty e^{-t}dZ(t)$, where $x>0$.  It is well known (see, e.g., \cite{Revuz-Yor-99}) that the stochastic differential equation 
$$\xi(t) =x + t + 2\int_0^t \sqrt{|\xi(s)|}dW(s) $$ has a unique strong solution $\xi^x$, and for all $t\ge 0$, $\xi^x(t) = x+ |W(t)|^2\ge 0$ if $x\ge 0$. 
Moreover, using \cite{SongYZ-12}, it follows that $\pr_0\set{\tau=0}=1$, where $\tau:=\inf\set{t> 0: \xi^0(t) > 0}$. Hence it follows that $\mathcal A_x \not= \emptyset$ for all $x \in [0,\infty)$.

The corresponding QVI is 
\begin{equation}
\label{qvi-ex4}
\min\set{u(x)-u'(x)-2x u''(x), u'(x)-1}=0, \ \  x\in (0,\infty). 
\end{equation}
One can easily check that $v(x)=x+1$ is a constrained viscosity solution to  \eqref{qvi-ex4}  on $[0,\infty)$. In fact, 
for $x>0$, 
\begin{displaymath}
\min\set{v(x)-v'(x)-2x v''(x), v'(x)-1}= \min \set{x+1-1 -2x\cdot 0, 1 -1}=\min\set{x, 0} =0.
\end{displaymath}
Moreover,  the subsolution property holds at the point $x=0$ since for any $\phi \in C^2$ with $(v-\phi)(x) \le (v-\phi)(0)=0$, 
we have $ \phi'(0) \ge 1$ and hence 
\begin{displaymath}
\min\set{\phi(0)-\phi'(0)-2\cdot 0 \cdot \phi''(0), \phi'(0)-1} = \min\set{1-\phi'(0), \phi'(0)-1} \le 0. 
\end{displaymath}
Thus $v(x)=x+1$ is indeed a constrained viscosity solution to  \eqref{qvi-ex4}  on $[0,\infty)$. 

Next, we demonstrate that \eqref{qvi-ex4} has at least another constrained viscosity solution on $[0,\infty)$.  First we note that the function 
$\psi(x):= \sinh(\sqrt{2x})$ is increasing and solves the equation $u(x)-u'(x)-2x u''(x)=0$ in $(0,\infty)$. Further, straightforward calculations reveal that 
\begin{displaymath}
\begin{aligned}
& \psi'(x) = \frac{ \cosh(\sqrt{2x})}{\sqrt{2 x}}, \ \ \psi''(x)=\frac{1}{2\sqrt 2}\left[ -\frac{\cosh(\sqrt{2x})}{x^{3/2}}+ \frac{\sqrt 2\sinh(\sqrt{2x})}{x}\right].
\end{aligned}
\end{displaymath}
The equation $\psi''(x)=0$ or equivalently $\frac{\cosh(\sqrt{2x})}{\sinh(\sqrt{2x})}= \sqrt{2x}$ has a unique positive  root, denoted by $z$.  Now we claim that the function defined by 
\begin{displaymath}
u(x)=  \sinh(\sqrt{2x}) \frac{\sqrt{2z}}{\cosh(\sqrt{2z})} I_{(0, z]}(x) + \( x-z+ \sinh(\sqrt{2z}) \frac{\sqrt{2z}}{\cosh(\sqrt{2z})}\) I_{ (z, \infty)}(x) 
\end{displaymath} 
is the only constrained viscosity solution to \eqref{qvi-ex4} on $[0,\infty)$. In fact, one can directly verify that $u(x)$ is  a solution to   \eqref{qvi-ex4} for $x>0$. 
As in Example \ref{exm1}, it remains to verify the subsolution property at the point $x=0$. To this end, let $\phi \in C^2$ with 
$(u-\phi)(x) \le (u-\phi)(0) =0$ for $x\in [0,\infty) $ in a neighborhood of $0$.  Then $\phi(0) =0$ and for $x>0$, $\phi(x) \ge u(x) >0$. Thus we must have $\phi'(0) \ge 0$ and hence 
\begin{displaymath}
\min\set{\phi(0) -\phi'(0) -0\cdot \phi''(0), \phi'(0)-1} \le 0.
\end{displaymath} 
The desired conclusion follows. Note that $u$ also satisfies the polynomial growth condition \eqref{uv-poly-growth}.

In terms of the singular control problem \eqref{BEDQ}, it turns out that the value function $V(x)= v(x)= x+1$. In fact, from the state constraint, we have 
$Z(t) \le x+ W(t)^2$ for any $t\ge 0$. Therefore 
\begin{displaymath}
\begin{aligned}
J(x,Z) & = \ex_x\int_0^\infty e^{-t} dZ(t) = \ex_x \int_0^\infty \int_t^\infty e^{-s}ds dZ(t) = \ex_{x} \int_0^\infty \int_0^s dZ(t) e^{-s}ds \\ 
& \le \ex_x \int_0^\infty 
 e^{-s} (x+ W^2(s))ds = \int_0^\infty e^{-s}(x+s)ds = x+1.
\end{aligned}\end{displaymath}
Furthermore, it is easy to check that the control $Z^*(t) =x+ W^2(t)$ is optimal and $J(x,Z^*)=x+1$. Hence $V(x)=x+1$ as claimed.
\hfill $\Box$\end{exm}

We finish the section with a hierarchical PDE characterization for the boundary behavior of the solution to \eqref{qvi-equiv-form}.
Let $\ell \subset\{1,2,\ldots, n\}$ be an index subset.
For a vector
$v= (v_1, \ldots v_n) \in \mathbb{R}^n$,
we induce a smaller vector 
$v^\ell := (v_i)_{i\in \ell} \in \mathbb{R}^{|\ell|}$, i.e.
$v_i^\ell = v_{\ell_i}$ for all $i\in \{1,\ldots, n\}$.
Typically, this notation will be used for
$v = b, \sigma, f, \xi, X, Z$.

In a reverse direction, for a vector $v\in \mathbb{R}^{|\ell|}$,
we define a larger vector
$v^{-\ell} \in \mathbb{R}^n$ by
$$(v^{-\ell})_j = \left\{
  \begin{array}{ll}
    v_i, & \hbox{ if } j = \ell_i, \\
    0, & \hbox{ otherwise}.
  \end{array}
\right.$$
For a function
$g: \mathbb{R}^n \times \mathcal{M} \mapsto \mathbb{R}^n$,
we induce another function
$g^\ell : \mathbb{R}^{\ell} \times \mathcal{M} \mapsto \mathbb{R}^n$
such that
$g_i^\ell(x,\alpha) = g_{\ell_i}(x^{-\ell}, \alpha).$

The following assumption is imposed.
\begin{itemize}
\item [(H1)] $b_i(x,\alpha) = \sigma_i(x,\alpha) = 0$
  on
  $\{x\in \overline{\mathbb{R}_+^n} \ |\ x_i = 0\}.$
\end{itemize}
This basically means that, in the content of ecosystem modeling, once  the $i$th species
becomes extinct,
it will never revive, i.e. if $(\zeta_i)_t = 0$ for some $t$, then
$(\zeta_i)_s = 0$ for all $s\ge t$.

Thanks to (H1), \eqref{sde-switching} implies
following sub-dynamics:
\begin{equation}
  \label{eq:sdel}
  d  \zeta^\ell(t) =
  b^\ell( \zeta^\ell(t), \alpha(t)) dt +
  \sigma^\ell( \zeta^\ell(t), \alpha(t)) dW^\ell(t),
  \  \zeta^\ell(0) = x^\ell, \alpha(0) = \alpha.
\end{equation}
Therefore, we can look at
following subsystem.
Suppose the survived species are indexed
by $\ell$ with its remaning
amount $x\in \mathbb{R}^{|\ell|}$, then the associated
value function can be defined as
$$J^\ell(x,\al,Z)
:=\ex \int_0^{\infty} e^{-r s}
f^\ell( X^{x,\al,\ell}(s-),\al(s-)) \cdot dZ^\ell(s)
$$
and 
$$V^\ell(x,\alpha) =
\sup_{Z\in \mathcal{A}_{x,\alpha}} J^\ell(x,\alpha,Z).$$

For all $x\in \mathbb{R}^{|\ell|}$, $\xi\in \mathbb{R}^m$,
$p\in \mathbb{R}^{|\ell|}$, $A\in \mathcal{S}^{|\ell|}$, we define
a function
$$
\begin{array}{ll}
  G^\ell(x,\alpha, \xi, p, A) =& 
  \min\{\displaystyle
  r - \xi_\alpha -
  \frac 1 2 tr(\sigma^\ell (\sigma^\ell)'(x^\ell, \alpha) A)
  - b^\ell(x^\ell, \alpha) \cdot p
  - \sum_{j=1}^m q_{ij} \xi_j, \\ & \hspace{2in}
  \displaystyle
  \min_{i= 1,\ldots, |\ell|} \{p_i - f^\ell_i(x,\alpha)\}
  \}
\end{array}
$$

Then, one can apply induction to the previous results to show that,
${\bf V}(x) = (V(x,\alpha))_\alpha$ is the
unique solution of
\begin{equation}
  \label{eq:pdel}
  \left\{
  \begin{array}{ll}
    G(x,\alpha, {\bf V}(x), DV(x,\alpha), D^2V(x,\alpha))
    = 0, 
    & (x,\alpha) \in \mathbb{R}_+^n\times \mathcal{M}\\
    V(x^{-\ell}, \alpha) = V^\ell(x,\alpha), &
    (x,\alpha) \in \mathbb{R}_+^{n-1}\times \mathcal{M},
    |\ell| = n-1.
  \end{array}\right.
\end{equation}

\section{Conclusions and Remarks}\label{sect-remarks}
In this work, we considered a class of singular control  problems with state constraints and regime-switching. 
The controlled dynamic is given by a regime-switching diffusion confined in the unbounded domain $S=\rr^n_+$ and the objective is to maximize the total expected  discounted rewards from exerting the singular control. Using the weak dynamic programming principle, we showed that the value function is the unique constrained viscosity solution of the system of coupled nonlinear quasi-variational inequalities \eqref{qvi-equiv-form}.  

Throughout our analysis, the discount rate $r$ was fixed. It is interesting to ask how the solution, with appropriate scaling of the cost,  will behave as $r \to 0$; and how the limit, if it exists, relates to that of the average cost control problem. A number of other questions deserve further investigations. In particular,  it is worth studying the case when the random environment or the Markov chain $\al$ is unobservable.


\def\cprime{$'$}

\end{document}